\documentclass[11pt,reqno]{amsart}

\newcommand\version{August 05, 2026}

\usepackage{esint}
\usepackage{amsthm}
\usepackage{amssymb}
\usepackage{mathrsfs}
\usepackage{graphicx}			
\usepackage{appendix}

\usepackage{stmaryrd}               
\usepackage{footmisc}               
\usepackage{paralist}
\usepackage{wrapfig}
\usepackage[draft]{pdfcomment}
\usepackage{bbm}
\usepackage{fontenc}
\usepackage{inputenc}
\usepackage[arrow, matrix, curve]{xy}
\usepackage{color}

\usepackage{tikz}
\usepackage{tikz-cd}

\usepackage{enumitem}

\usepackage{mathtools}

\newtheorem{thm}{Theorem}[section]

\newtheorem{lem}[thm]{Lemma}

\newtheorem{cor}[thm]{Corollary}
\newtheorem{prp}[thm]{Proposition}

\newtheorem{refdef}[thm]{Lemma on Definition}

\theoremstyle{remark}
\newtheorem{remsthm}[thm]{Remarks on Theorem}

\newtheorem{remsdfn}[thm]{Remarks on Definition}
\newtheorem{remslem}[thm]{Remarks on Lemma}

\theoremstyle{definition}
\newtheorem{dfn}[thm]{Definition}

\newtheorem*{qst}{Question}

\theoremstyle{remark}
\newtheorem{rem}[thm]{Remark}

\newtheorem{exl}{Example}

\newcommand{\R}{\mathbb{R}}
\newcommand{\N}{\mathbb{N}}

\newcommand{\HH}{\mathbb{H}}

\newcommand{\Lo}{\mathbb{L}}

\newcommand{\Cauchy}{\mathcal{C}}

	\usepackage[nodayofweek, level]{datetime}

\title[Hausdorff-type metric geometry of the space of Cauchy hypersurfaces]{Hausdorff-type metric geometry of \\ the space of Cauchy hypersurfaces 
}

\author[Christian Lange]{Christian Lange}
\address[Christian Lange]{Ludwig-Maximilians-Universit\"at M\"unchen, Mathematisches Institut\newline\indent Theresienstr. 39, 80333 München, Germany}
\email{lange@math.lmu.de}

\author[Jonas W. Peteranderl]{Jonas W. Peteranderl}
\address[Jonas W. Peteranderl]{Ludwig-Maximilians-Universit\"at M\"unchen, Mathematisches Institut\newline\indent Theresienstr. 39, 80333 München, Germany}
\email{peterand@math.lmu.de}

\subjclass[2020]{Primary: 53C23, 51K10, 53C50. Secondary: 53C80, 28A75, 53B30, 83C99.}

\keywords{Lorentzian geometry, Cauchy hypersurfaces, Hausdorff metric, completeness}
\date{\version}
\thanks{\copyright\, 2026 by the authors. This paper may be reproduced, in its entirety, for non-commercial purposes.}

\begin{document}

\begin{abstract}
We equip the space of Cauchy hypersurfaces in a globally hyperbolic spacetime with a natural Hausdorff-type metric. For a timelike Cauchy complete, smooth Lorentzian manifold with a compact Cauchy hypersurface, we show that the resulting metric space is geodesic and proper. We further discuss extensions to more general synthetic Lorentzian settings. For this purpose, we generalize results on completeness properties of spacetimes due to Beem and Takahashi.
\end{abstract}

\maketitle

 \section{Introduction and main results}

A Cauchy hypersurface can be defined as a subset of a spacetime that is intersected exactly once by every inextendible timelike curve, or more intuitively, by the wordline of every physical observer. As such, Cauchy hypersurfaces represent ``instants in time'' and arise, for instance, when solving the Einstein field equations with prescribed initial data. By a landmark result of Geroch \cite{Ge70}, their existence in a given smooth Lorentzian manifold characterizes the strongest and most important causality condition in general relativity: global hyperbolicity. This condition is essential, for instance, in the singularity theorems of Hawking and Penrose \cite{HP70,HE73} as well as for Lorentzian splitting theorems as developed in \cite{Es88,Ga89}. 

Motivated by questions from Lorentzian dynamics and by the lack of a canonical choice of a Cauchy hypersurface in a given spacetime, Monclair \cite{Mo23} recently initiated the study of the space of Cauchy hypersurfaces by analyzing a weak Riemannian metric, defined in terms of an $L^2$-inner product, on an infinite-dimensional manifold of compact spacelike Cauchy hypersurfaces in a smooth, spatially compact, globally hyperbolic Lorentzian manifold. In particular, Monclair shows that the induced path distance of this weak Riemannian metric is not only a pseudometric, but an actual metric \cite[Theorem~1.1]{Mo23}.

In this work we endow the space of all Cauchy hypersurfaces in a given spacetime with a natural Hausdorff-type metric, which can be regarded as an $L^\infty$-Finsler analog of Monclair's metric. While our metric is not Riemannian anymore, it has the advantage that it does not rely on smoothness assumptions so that it can also be defined in abstract synthetic settings; see Section~\ref{sec:lorentzian_spaces}.

For simplicity, we first give a more precise description in the setting of a smooth, globally hyperbolic Lorentzian manifold $(M,g)$. In this case, the induced Lorentzian distance $d$ on $M$ is finite and continuous; see \cite[Lemma~4.5]{BEE96}, for instance.
In the special case of Minkowski space $M=\Lo^{n+1}$, Bahn and Ehrlich \cite{BE99} defined
\[
    d_J(A,B) \coloneqq  \sup_{x\in A, y\in B} (d(x,y) +d(y,x))
\]
for two subsets $A,B \subset M$ and observed some formal resemblance with the Hausdorff metric. However,  
the function $d_J$ in general fails all properties of a metric except symmetry and non-negativity. Nevertheless, for achronal smooth spacelike hypersurfaces with boundary $S_1$ and $S_2$ that are sufficiently causally related it is shown in \cite[Lemma~5]{BE99} that $d_J(S_1,S_2)=0$ if and only if $S_1=S_2$. While a Cauchy hypersurface is in general not smooth, we show that $d_J$, restricted to the space $\Cauchy_M$ of Cauchy hypersurfaces in $M$, satisfies all properties of an \emph{extended} metric, i.e.~a metric that is allowed to take the value infinity. In fact, if there exists a compact Cauchy hypersurface in $M$, then all Cauchy hypersurfaces in $M$ are compact by \cite[Property 7]{Ge70}, and in this case the metric is finite.
Moreover, in case $(M,g)$ satisfies additional completeness conditions, we show that $\Cauchy_M$ is contained in a larger complete extended metric space $\Cauchy_M^w$ consisting of a priori more general ``weak Cauchy sets'', as considered in Definition~\ref{dfn:cauchy_set}, and
establish further properties for these (extended) metric spaces. More precisely, we show the following statements. Recall that an (extended) metric space $(X,d)$ is called \emph{geodesic} if any pair of points $p,q \in X$ (with $d(p,q)<\infty$) can be connected by a curve of length $d(p,q)$, \emph{locally compact} if every point has a compact neighborhood, and \emph{proper} if closed balls are compact.

\begin{thm}\label{thm:metric_properties} Let $(M,g)$ be a smooth, globally hyperbolic Lorentzian manifold. Then the space of Cauchy hypersurfaces $(\Cauchy_M,d_J)$ in $M$ is a (non-empty) extended metric space. Moreover, if $(M,g)$ is
\begin{enumerate}
\item timelike geodesically complete, then $(\Cauchy_M^w,d_J)$ is a complete extended metric space.
\item timelike Cauchy complete, then $(\Cauchy_M,d_J)$ is a complete extended metric space, 
\item spatially compact, then $(\Cauchy_M,d_J)$ is a locally compact metric space.
\item timelike Cauchy complete and spatially compact, then $(\Cauchy_M,d_J)$ is a proper geodesic metric space.
\end{enumerate}
\end{thm}

\addtocounter{thm}{-1}
\begin{remsthm}
(a) The assumptions timelike geodesically complete (see Subsection~\ref{sub:ms_spaces}) in (i) and  timelike Cauchy complete in (ii) are necessary; see Example ~\ref{exl:non_geod_complete_examples}. The condition ``timelike Cauchy complete'' is discussed in Section~\ref{sec:completeness}. Here we only mention that according to
    \cite[Lemma~9]{Se67} and \cite[Corollary~6]{Be76} (see also \cite[Theorem~6.12]{BEE96})
  every smooth, causally (i.e.~timelike and null)
 geodesically complete, globally hyperbolic Lorentzian manifold is timelike Cauchy complete, and that timelike Cauchy completeness implies that every weak Cauchy set is a Cauchy set; see Corollary~\ref{cor:weak cauchy_is_cauchy}.

    (b) The conclusion of local compactness in (iii) can be regarded as an analog of Blaschke's selection theorem for convex sets \cite[p.~62]{Bl16}. The assumption that $M$ is spatially compact, i.e.~that it has a compact Cauchy hypersurface (see Subsection~\ref{subsec:Cauchy_sets} for more details),  
    is necessary; see Subsection~\ref{sub:local_cpct}. 

    (c) For an application of the metric $d_J$ in the context of Hausdorff stability estimates for Lorentzian isoperimetric inequalities, we refer to \cite{LP25}. 
    \end{remsthm}

While our proof of Theorem~\ref{thm:metric_properties}, (i), uses the geodesic equation and our proof of the compactness statements in Theorem~\ref{thm:metric_properties}, (iii) and (iv), uses the Bernal--S\'anchez splitting theorem \cite{BS03} (see Subsection~\ref{sub:local_cpct}), we show variants of the other statements of Theorem~\ref{thm:metric_properties} in a more abstract setting, which comprises synthetic Lorentzian spaces in the sense of  Bykov--Minguzzi--Suhr \cite{MS24,BMS25}, Braun--McCann \cite{BM26}, and Kunzinger--Sämann \cite{KS18}, respectively; see Corollary~\ref{cor:metric_corollary}, Corollary~\ref{cor:compact_finite}, Proposition~\ref{prp:conditions_compl_metric}, and Corollary~\ref{cor:geodesic}. Such synthetic generalizations of Lorentzian manifolds in the spirit of metric geometry have been introduced with the motivation to identify minimal requirements for the logical dependence among fundamental results of Lorentzian geometry and causality theory (see \cite{KS18}) as well as to provide a mathematical framework for Lorentzian geometry and physical theories in low regularity with potential singularities as encountered, e.g.~in the presence of matter discontinuities and black holes, and in view of quantum gravity theories. For instance, we show that for a (timelike Cauchy complete) spatially compact, globally hyperbolic Kunzinger--Sämann Lorentzian length space $X$, the space of Cauchy sets $(\mathcal{C}_X,d_J)$ is a (complete, geodesic) metric space. Note that if $X$ is in addition second countable (and proper), then the existence of a Cauchy set in $X$ follows from a construction of a Cauchy time function on $X$ due to Minguzzi \cite[Theorem 21]{Mi26} (and Burtscher--Garc\'ia-Heveling \cite[Theorem 1.3]{BG25}, respectively); see Subsection~\ref{subsec:Cauchy_sets}. To provide further examples to which the statements of Theorem~\ref{thm:metric_properties} apply (see Corollary~\ref{cor:metric_corollary}), we explain in Appendix~\ref{app:time_functions} how Minguzzi's construction can be modified to also work for spacetimes with a non-empty chronological boundary.
\vspace{0.2cm}

For the discussion in this more general setting, we moreover generalize results by Beem \cite{Be76} and Takahashi \cite{Ta25} on completeness properties of globally hyperbolic spacetimes. Namely, 
Beem \cite{Be76} describes a completeness property of globally hyperbolic spacetimes in terms of the different equivalent conditions ``finite compactness'', ``timelike Cauchy completeness'', and ``Condition A''; see Section~\ref{sec:completeness}. Roughly speaking, this property is satisfied if the only way to ``escape to infinity'' while staying in finite Lorentzian distance to a point is to do so in the null cone of this point. 

Takahashi proves that in a globally hyperbolic Kunzinger--Sämann Lorentzian length space finite compactness implies timelike Cauchy completeness and that the latter implies Condition~A \cite[Theorem 3.9, (1) and (2)]{Ta25}. We further relax the assumptions of the first implication to ``causally simple Lorentzian pre-length space'', cf.~Section~\ref{sec:lorentzian_spaces}. In addition, we show a refinement of the second implication under relaxed assumptions. This involves the alternative completeness conditions, ``Condition B'' and ``Condition A$^*$'' as defined in Section~\ref{sec:completeness}, which are relevant for our proof of Theorem~\ref{thm:metric_properties}  and stronger than Condition~A in the settings considered in \cite{Be76,Ta25}. Moreover, Theorem \ref{thm:complete}  states that these implications do not only hold for Lorentzian pre-length spaces in the sense of Kunzinger--Sämann, but also in settings without a background metric as, for instance, in \cite{MS24,BM26,MS26}; see Remarks on Definition~\ref{dfn:inext_def} for a comparison and Section~\ref{sec:lorentzian_spaces} for definitions of these spaces.

\begin{thm}\label{thm:complete} Let $X$ be a first-countable, causally simple Lorentzian 
pre-length space.
\begin{enumerate}
\item If $X$ is finitely compact, then it is timelike Cauchy complete.
\item If $X$ is timelike Cauchy complete, then it satisfies Condition~B.
\end{enumerate} 
Moreover, for a globally hyperbolic Kunzinger--Sämann Lorentzian length space, a separable, timelike path-connected Bykov--Minguzzi--Suhr Lorentzian metric space without chronological boundary, and a separable, timelike path-connected, globally hyperbolic Braun--McCann length metric spacetime, in addition to (i) and (ii), Condition~B implies Condition~A$^*$, and Condition~A$^*$ implies Condition~A.
\end{thm}

In particular, for a smooth, globally hyperbolic Lorentzian manifold the conditions finite compactness, timelike Cauchy completeness, Condition~B, Condition~A$^*$, and Condition~A are all equivalent, since Condition~A implies finite compactness in this case; see \cite[Theorem~5]{Be76} and Corollary \ref{cor:completeness_smooth_case} for a refined statement. This refinement concerns another interpretation of Conditions B, A$^*$, and A in case the topology is induced by a metric, namely in terms of inextendible \emph{locally Lipschitz continuous} causal curves instead of inextendible \emph{continuous} causal curves; see Section \ref{sec:inext}. In the former sense, Theorem \ref{thm:complete} holds true for Kunzinger--Sämann Lorentzian pre-length resp. length spaces under a necessary (see Remarks on Definition~\ref{dfn:inext_def}, (e)) additional metric-compatibility assumption in statement (ii) (see Proposition \ref{prp:tcc_propB}), which is satisfied by a Kunzinger--S\"amann Lorentzian length spaces by definition. Further note that, unlike \cite[Theorem 3.9]{Ta25}, the assumptions of the first part of Theorem~\ref{thm:complete} also allow the possibility of a non-empty chronological boundary. The stronger assumptions in the second part of Theorem~\ref{thm:complete} are required partly to guarantee the existence of enough inextendible timelike curves (see Corollary \ref{cor:inex_timelike}) as used for Condition B; see Remark \ref{rem:boundary_version_prop_B} though for a variant of this statement that allows a non-empty chronological boundary. In the BMS setting the assumptions separable and no chronological boundary are equivalent to the assumption ``countably generated'' from \cite{BMS25}; see Subsection \ref{sub:BMS_spaces}.

Finally, the impact of the conditions in Theorem~\ref{thm:complete} is further elucidated by Proposition~\ref{prp:geodesic_time_function}, in which we show how Condition B allows an easy construction of a Cauchy time function starting from a Cauchy set and using geodesic rays with respect to our Hausdorff-type metric on the space of Cauchy sets.

\vspace{0.2cm}

\textbf{Structure of the paper.} In Section~\ref{sec:lorentzian_spaces} we recall basic notions from Lorentzian geometry and give an overview on different approaches to synthetic Lorentzian geometry. Following this, we discuss several notions of inextendible causal curves in Section~\ref{sec:inext} and prove some basic related existence results, which underlie our main theorems. In Section~\ref{sec:completeness} we establish relations between various completeness conditions and thereby prove Theorem~\ref{thm:complete}. After introducing different notions of Cauchy sets and discussing their existence and
properties, we finally prove Theorem~\ref{thm:metric_properties} 
in Section~\ref{sec:hausdorff_cauchy}.
\vspace{0.4cm}

\textbf{Acknowledgements.} The authors are grateful to Annegret Burtscher, Melanie Graf, Ettore Minguzzi, Clemens Sämann, and Stefan Suhr for answers to different questions, useful comments, and hints to the literature. The second author acknowledges partial support through the Deutsche Forschungsgemeinschaft (DFG, German Research Foundation) grants FR 2664/3-1 and TRR 352-Project-ID 470903074 and through the Studienstiftung des deutschen Volkes.

\section{Lorentzian geometry preliminaries}\label{sec:lorentzian_spaces}

To start with, we recall some basic notions from classical Lorentzian geometry.
For extensive introductions, we refer, for instance, to \cite{BEE96,On83}. Throughout, we assume $(M,g)$ to be a Lorentzian manifold, which we define as a smooth manifold with a smooth Lorentzian metric, unless stated otherwise. A tangent vector $v$ of $M$ is called \emph{causal} if $g(v,v) \leq 0$, \emph{timelike} if $g(v,v)<0$, and \emph{null} if $g(v,v)=0$. A \emph{time-orientation} on $M$ is a continuous choice of a connected component of the set of causal tangent vectors without the origin in the tangent spaces of $M$. We suppose that $(M,g)$ is time-oriented and call a causal tangent vector \emph{future-directed} if it belongs to the chosen component. A locally Lipschitz continuous curve in $M$ is called \emph{future-directed causal} if almost all its tangent vectors are future-directed causal. 

Let $p,q\in M$. The \emph{causal relation} $\leq$ on $M$ is defined by setting $p \leq q$ if there exists a future-directed causal curve from $p$ to $q$. We write $|v|= \sqrt{-g(v,v)}$ for a causal tangent vector $v$. The \emph{Lorentzian distance} $d(p,q)$ is defined as the supremum of the Lorentzian lengths $\int | \dot \gamma(t)| \,\mathrm dt$ over all causal curves from $p$ to $q$, if $p\leq q$, and as $d(p,q)=0$ otherwise. The \emph{timelike relation} $p \ll q$ denotes $d(p,q) > 0$.

Before we continue to discuss the causal structure and the Lorentzian geometry of $(M,g)$ at the end of Subsection~\ref{sub:ms_spaces}, we first describe an axiomatization in the synthetic spirit of \cite{KS18,BM26,MS26} to provide a more general framework. 
Early works in this direction  are due to Busemann \cite{Bu67} and to Kronheimer--Penrose \cite{KP67}. From Subsection~\ref{sub:BMS_spaces} on, which can be skipped on first reading, we embed and compare the approaches from
\cite{KS18,BM26,MS26} in the framework of 
Subsection~\ref{sub:ms_spaces}, focusing on those parts of the theory that are relevant for our work.

\subsection{Synthetic Lorentzian geometry} \label{sub:ms_spaces}
We first recall the concept of a causal space in the sense of Kunzinger--Sämann \cite{KS18}, which is inspired by the idea of a Kronheimer--Penrose causal space \cite{KP67}. 

Let $X$ be a set, $\leq$ a reflexive and transitive relation on $X$, and $\ll$ a transitive relation on $X$ contained in $\leq$. Then $(X,\leq,\ll)$ is called a \emph{causal space} with \emph{causal relation} $\leq$  and \emph{timelike relation} $\ll$. Natural examples are provided by Lorentzian manifolds as above.

For a causal space $(X,\leq,\ll)$ the \emph{chronological future} and \emph{past} of $x\in X$ are defined by
\begin{equation*}
    I^+(x) \coloneqq \{y \in X \mid x \ll y \} \quad \text{and}\quad  
    I^-(x) \coloneqq \{y \in X \mid y \ll x \} \,,
\end{equation*}
 respectively. Similarly, the \emph{causal future} and \emph{past}  of a point $x\in X$ are defined by
\begin{equation*}
        J^+(x) \coloneqq \{y \in X \mid x \leq y \} \quad \text{and}\quad
    J^-(x) \coloneqq \{y \in X \mid y \leq x \} \,,
\end{equation*} respectively.
We write $I^\pm(A)= \bigcup_{x\in A}I^\pm(x)$ and $J^\pm(A)= \bigcup_{x\in A}J^\pm(x)$ for a subset $A \subset X$.

A causal space $(X,\leq,\ll)$ is called \emph{chronological} if $\ll$ is irreflexive, i.e.~if $x\not \ll x$ for all $x \in X$. The following boundary notions for a chronological causal space $(X,\leq,\ll)$ occur below but can be skipped by those only interested in Theorem~\ref{thm:metric_properties}. The \emph{future} resp.~\emph{past causal boundary}
\[
     \partial^\pm_{ca} X \coloneqq \{x\in X \mid J^\pm(x)= \{x\}\}
\]
of $X$ is contained in the \emph{future} resp.~\emph{past chronological boundary}
\[
        \partial^\pm_{ch} X \coloneqq \{x\in X \mid I^\pm(x)=\emptyset\} 
\]
of $X$. The \emph{chronological} resp.~\emph{causal boundary} of $X$ is defined as $\partial_{ch} X \coloneqq \partial^-_{ch} X \cup \partial^+_{ch} X$  resp.~$\partial_{ca} X \coloneqq \partial^-_{ca} X \cup \partial^+_{ca} X$. We say that $X$ has \emph{no bubbling boundary} if $\partial^+_{ca} X = \partial^+_{ch} X$ and $\partial^-_{ca} X = \partial^-_{ch} X$; cf.~\cite{GKSS20,GS22}. In this case we simply write $\partial^+X\coloneqq \partial^+_{ca} X = \partial^+_{ch} X$ and $\partial^-X\coloneqq \partial^-_{ca} X = \partial^-_{ch} X$. The \emph{spacelike boundary} of $X$ is defined as $\partial^+_{ch}X \cap \partial^-_{ch}X$.

\vspace{0.2cm}
We now follow the approaches of Braun--McCann \cite{BM26} and Mondino--Sämann \cite{MS26} to encode a causal structure and a Lorentzian distance in a single function.

An \emph{extended time separation function} on a set $X$ is a map $\ell : X \times X \rightarrow \{-\infty\} \cup [0,\infty]$ that satisfies $\ell(x,x) \geq 0$ for all $x\in X$ and the \emph{(extended) reverse triangle inequality}
\begin{equation}
\label{eq:extended_reverse_triangle}
 \ell(x,y) + \ell(y,z) \leq \ell(x,z) \qquad \text{for all } x,y,z \in X \, ,
\end{equation}
with the convention that $-\infty + \infty=- \infty$. Then $\ell(x,x) \in \{0,\infty\}$ for all $x\in X$, and the associated \emph{Lorentzian distance} $d \coloneqq  \max(0,\ell)$ (also referred to as \emph{time separation function} $\tau$ in \cite{KS18,MS26}) satisfies the reverse triangle inequality \eqref{eq:extended_reverse_triangle} for all $x,y,z \in X$ with $x\leq y\leq z$. The corresponding \emph{causal relation} $\leq$ and \emph{timelike relation} $\ll$ are defined as
\[
    J\coloneqq \leq \coloneqq \ell^{-1}([0,\infty])\quad \text{and}\quad  I\coloneqq I_d \coloneqq \ll \coloneqq d^{-1}((0,\infty]) \subset J \,,
\]
respectively. By the reverse triangle inequality $(X,\leq,\ll)$ is a (chronological if $\ell < \infty$) causal space and the \emph{push-up principle} $J\circ I \cup I \circ J \subset I$ holds, where $J \circ I \coloneqq\{(x,z)\in X \times X \mid \exists y\in X: x\leq y \ll z\}$ and $I\circ J$ is defined analogously. The extended time separation function $\ell$ can be recovered from the Lorentzian distance $d$ and the causal relation via
\begin{align} \label{eq:recover_etsf}
\ell(x,y) = \begin{cases} d(x,y) &\text{if } x\leq y \,,\\
                    - \infty &\text{else\,.}
\end{cases} 
\end{align}
In fact, for any map $d:X \times X \rightarrow [0,\infty]$ and any reflexive and transitive $J \subset X \times X$ with $I_d \subset J$ and $d$ satisfying the reverse triangle inequality \eqref{eq:extended_reverse_triangle} for all $x,y,z \in X$ with $(x,y),(y,z) \in J$, the function $\ell$ defined by \eqref{eq:recover_etsf} is an extended time separation function with associated Lorentzian distance $d$ and causal relation $J$. For instance, this provides a natural extended time separation function in case of a Lorentzian manifold discussed above.

Equipped with a topology with respect to which $I^{\pm}(x)$ is open for all $x\in X$, the pair $(X,\ell)$ is a \emph{Lorentzian pre-length space} in the sense of \cite[Definition~2.3]{MS26}. For instance, a  Lorentzian manifold gives rise to such a Lorentzian pre-length space by lower semicontinuity of its Lorentzian distance; see \cite[Lemma~14.17]{On83}, for instance. Under further assumptions the pair $(X,\ell)$ could be referred to as a ``metric spacetime'' in the sense of \cite{BM26}; cf.~Subsection~\ref{sub:BM_spaces}.

Physically meaningful spacetimes satisfy additional causality conditions.   (The underlying topological preordered space of) a Lorentzian pre-length space $(X,\ell)$ is called
\begin{enumerate}
\item \emph{causal}, if $\leq$ is antisymmetric, i.e.~if $x \leq y \leq x$ implies $x=y$ for all $x,y\in X$, 
\item \emph{causally simple}, if it is causal and the sets $J^{\pm}(x)$, $x\in X$, are closed,
\item \emph{globally hyperbolic}, if it is causal, the causal relation is closed, and for every compact subset $K\subset X$ the \emph{causal emerald} $J(K)\coloneqq J^+(K) \cap J^-(K)$ is compact,
\end{enumerate}
and the following hierarchy holds: (iii)$\Rightarrow$(ii)$\Rightarrow$(i); see \cite[p.~2]{Mi23}. Note that in case of (iii) also all \emph{causal diamonds} $J(x,y)\coloneqq J^+(x) \cap J^-(y)$ are compact. Further note that $J(x,y)=J(\{x,y\})$ for $x\leq y$ and $J(x,y)=\emptyset$ otherwise. Global hyperbolicity, arguably the strongest (and most well-known) causality condition that can be imposed, has several equivalent characterizations in different settings; see \cite{Ge70, BS03, BS07, Mi23,BG24, BG25, BM25} to name a few. This notion will be discussed further below in the setting of more specific Lorentzian pre-length spaces. 

\vspace{0.2cm}

In the following, all curves in a Lorentzian pre-length space $(X,\ell)$ are assumed to be continuous and defined on an interval $I \subset \R$, which should not be confused with the chronological relation $I$. (In the setting of KS-Lorentzian pre-length spaces discussed in Subsection~\ref{sub:ks_spaces}, all curves are assumed to be locally Lipschitz continuous, and the following notions are understood in terms of locally Lipschitz continuous curves.) A curve $\gamma: I \rightarrow X$ is called future-directed \emph{causal} resp.~ \emph{timelike} if $t<s$ for $t,s \in I$ implies $\gamma(t) \leq \gamma(s)$ resp.~$\gamma(t) \ll \gamma(s)$. In the following, we omit the specification future-directed, i.e.~all causal curves are assumed to be future-directed unless stated otherwise. We call a Lorentzian pre-length space $X$ \emph{causally} resp.~\emph{timelike path-connected} if every pair of points  $x,y\in X$ with $x\leq y$ resp.~$x\ll y$ can be connected by a causal resp.~timelike curve. In contrast, note that a Lorentzian pre-length space is called causally path-connected in \cite{KS18} if it is causally and timelike path-connected in our sense. The \emph{Lorentzian length} of a causal curve $\gamma:[a,b] \rightarrow X$ is defined as
\begin{equation*}
L(\gamma) \coloneqq \inf \left\{ \sum_{i=0}^{N-1} d(\gamma(t_i),\gamma(t_{i+1})) \mid N \in \N, a=t_0 < t_1 <\dots < t_N=b \right\} \, .
\end{equation*}
For more general
intervals $I$, the length is defined as the supremum of the length of $\gamma_{|I'}$ over all compact intervals $I' \subset I$. For locally Lipschitz continuous curves in a Lorentzian manifold, these notions of a causal curve and its Lorentzian length coincide with the previous ones.

The Lorentzian pre-length space $X$ is called \emph{geodesic}  if for every pair of points $x,y\in X$ with $x \leq y$ there exists a \emph{maximal} causal curve $\gamma$ from $x$ to $y$, i.e.~a causal curve $\gamma$ from $x$ to $y$ such that $d(x,y) = L(\gamma)$. Adapting the notion in \cite[Definition~2.41]{BM26}, we call $X$ \emph{regular} if every injective maximal causal curve connecting timelike related points in $X$ is timelike. If $X$ is geodesic and regular, then it is also timelike path-connected. 

\vspace{0.2cm}

We continue to discuss the above notions in the context of our prime example of a Lorentzian pre-length space, namely the one induced by a  Lorentzian manifold $(M,g)$. It is causally path-connected by definition. By the push-up principle it is also timelike path-connected; see, for instance, \cite[Proposition 2.4.18]{Ch11}. That it is regular (with respect to locally Lipschitz continuous curves) is, for instance, shown in \cite[Theorem~3.18]{KS18}. In fact, every maximal curve in $M$ admits a reparametrization that satisfies the geodesic equation (see e.g.~\cite[Proposition~5.34]{On83}) and hence has constant velocity. If it connects timelike related points, then it must thus be timelike even in the stronger sense that (almost all) its derivatives are timelike. Note that, for instance, the smooth curve $t\mapsto (t, \sin t, \cos t)$ in Minkowski space $\Lo^{1+2}$ is timelike only in the weaker sense of Subsection~\ref{sub:ms_spaces}. However, this distinction does not really make a difference in the following when working on smooth, globally hyperbolic Lorentzian manifolds; see Remark~\ref{rem:timelike_variation}. While the Lorentzian distance between two causally related points can by definition be approached by the Lorentzian lengths of connecting causal curves, the existence of maximal connecting curves requires 
further assumptions. By the Avez--Seifert theorem \cite{Av63,Se67},  global hyperbolicity is a sufficient condition; see e.g.~\cite[Proposition~14.19]{On83}. Note that for Lorentzian manifolds global hyperbolicity as defined above is equivalent to causality and compactness of causal diamonds and further coincides with the traditional one as in \cite{Ge70, HE73,Be76}, for instance; see \cite{Mi23}. 

For later purposes, we recall here that a smooth Lorentzian manifold $(M,g)$ is called \emph{null} resp.~\emph{timelike geodesically complete} if every null resp.~timelike solution of the geodesic equation can be extended to all of~$\R$.

In the following subsections, which can be skipped on first reading, we discuss and compare more specific classes of Lorentzian pre-length spaces.

\subsection{Lorentzian metric and length spaces \`a la Bykov--Minguzzi--Suhr} \label{sub:BMS_spaces}

Based on properties of globally hyperbolic Lorentzian manifolds as given in \cite[Theorem 1.1]{BM25}, Bykov, Minguzzi, and Suhr \cite{MS24,BMS25} provide 
a synthetic definition of a spacetime that is inherently ``globally hyperbolic''. In \cite[Definition~2.1]{BMS25} they define a \emph{Lorentzian metric space} 
as a set $X$ with a function $d:X\times X \rightarrow [0,\infty)$, called the Lorentzian distance, which satisfies the following properties
\begin{enumerate}
\item $d$ satisfies the reverse triangle inequality for all $x,y,z\in X$ with $d(x,y)>0$, $d(y,z)>0$.
\item There exists a topology on $X$ for which $d$ is continuous with respect to the product topology and for which the sets $\{(x',y')\mid d(x',y')\geq \varepsilon\}\cap (\overline{I(x,y)}\times \overline{ I(x,y)}) $ with $x,y\in X$, $\varepsilon>0$, and $I(x,y)\coloneqq I^+(x)\cap I^-(y)$ are compact.

\item $d$ distinguishes points in the sense that
\begin{equation*}
\forall x\neq y \, \,\,\, \exists z  : d(x,z)\neq d(y,z) \text{ or } d(z,x) \neq d(z,y)
\, .
\end{equation*}
\end{enumerate}
The timelike relation $I=I_d$ is defined as in Subsection~\ref{sub:ms_spaces}, so it only depends on the Lorentzian distance. By condition (iii),  the spacelike boundary of a Lorentzian metric space consists of at most one point. If the chronological boundary of $X$ is empty, then the topology satisfying (ii) is (Hausdorff, locally compact and) unique \cite[Theorem~7.2]{BMS25}. In the following, we always think of a Lorentzian metric space as equipped with some topology satisfying (ii). 

A Lorentzian metric space is called \emph{countably generated} if there exists a countable set $\mathscr{S}$ such that for each $y\in X$ there are $x,z \in \mathscr{S}$ with $x \ll y \ll z$. Note that the chronological boundary of a countably generated Lorentzian metric space is empty. Moreover, it is second countable, $\sigma$-compact, and a Polish space, i.e.~separable and completely metrizable; see \cite[Theorem~7.2]{BMS25}. Conversely, by continuity of the Lorentzian distance, every separable Lorentzian metric space without chronological boundary is countably generated. 

By \cite[Theorem~4.4]{BMS25} and the discussion following \eqref{eq:recover_etsf}, a Lorentzian metric space can be regarded as a chronological and causal Lorentzian pre-length space $(X,\ell)$ with respect to its \emph{maximal causal relation}
\begin{equation*}
 J_d\coloneqq \{(x,y) \in X\times X \mid  d(z,x) \leq d(z,y) \text{ and } d(y,z) \leq d(x,z) \text{ for all }z\in X\} \, ,
\end{equation*} 
which is in addition closed by continuity of $d$; see \cite[Theorem~4.4]{BMS25}. In particular, 
the associated Lorentzian pre-length space $(X,\ell)$ is causally simple; see \cite[p.~2]{Mi23}. Moreover, if the chronological boundary of $X$ is empty, then $(X,\ell)$ is globally hyperbolic by \cite[Theorem 4.5]{BMS25}. Note that the compactness condition in (ii) is implied by compactness of all $J_d$-causal diamonds if we assume all the other defining properties. 

A \emph{Lorentzian length space} in the sense of \cite[Definition~5.1]{BMS25} is a geodesic Lorentzian metric space. We refer to it as a \emph{BMS-Lorentzian length space} in the following.

\subsection{Length metric spacetimes \`a la Braun--McCann} \label{sub:BM_spaces}
A length metric spacetime in the sense of Braun--McCann \cite{BM26} is a first-countable, causal Lorentzian pre-length space $(X,\ell)$ with the following additional properties.
\begin{enumerate}
\item $\ell$ is upper semicontinuous, and the associated Lorentzian distance $d$ is lower semicontinuous. 
\item Every $x\in X$ is the midpoint of a timelike  curve $\gamma:[0,1] \rightarrow X$, i.e.~$x=\gamma(1/2)$.
\item The Lorentzian distance of timelike related points $x \ll y$ is realized as the supremum of the Lorentzian lengths of (injective) causal curves from $x$ to $y$.
\end{enumerate}
We refer to such a Lorentzian pre-length space as a \emph{BM-length metric spacetime}. By condition (ii), the chronological boundary of a BM-length metric spacetime is empty.

By \cite[Corollary~5.27]{BMS25} and \cite[Theorem 2.13]{BM26}, the Lorentzian pre-length spaces associated (as in the previous section) to BMS-Lorentzian length spaces with the additional property that every point is the midpoint of a timelike curve coincide with globally hyperbolic BM-length metric spacetimes. In particular, a BM-length metric spacetime is geodesic; cf.~\cite[Theorem~2.37]{BM26}.

\subsection{Lorentzian pre-length and length spaces \`a la Kunzinger--S\"amann}\label{sub:ks_spaces}

Kunzinger and Sämann define a Lorentzian pre-length space as a Lorentzian pre-length space in the sense of Subsection~\ref{sub:ms_spaces}  equipped with a metric $\rho$ that induces its topology and with respect to which the Lorentzian distance $d$ is lower semicontinuous; see \cite[Definition~2.8]{KS18}. We refer to such a Lorentzian pre-length space as a \emph{KS-Lorentzian pre-length space}. In this setting all curves and all notions introduced at the end of Subsection~\ref{sub:ms_spaces} are understood in terms of locally Lipschitz continuous curves. A KS-Lorentzian pre-length space that is ``regular(ly localizable)'' in the sense of \cite[Definition 3.16]{KS18} is regular in the sense of Subsection~\ref{sub:ms_spaces}; cf.~\cite[Theorem~3.18]{KS18}.

The notion of a KS-Lorentzian pre-length space underlies another notion introduced by Kunzinger and Sämann \cite{KS18}. Roughly speaking, they define a \emph{Lorentzian length space} as a causally and timelike path-connected KS-Lorentzian pre-length space for which $d(x,y)$ for $x\ll y$ is always realized as the supremum of the Lorentzian lengths of causal curves from $x$ to $y$ and for which further causality and localizability properties motivated by Lorentzian manifolds are satisfied; see \cite[Definition 3.22]{KS18} for the details. 

In particular, Lorentzian manifolds, even with metrics of low regularity (see \cite[Theorem~5.12]{KS18}), give rise to a Lorentzian length space. However, (strong) regular localizability only holds if they are sufficiently smooth; cf.~\cite[Example 3.24]{KS18}. We refer to a Lorentzian length space in the sense of Kunzinger--Sämann as a \emph{KS-Lorentzian length space} in the following. 

By timelike path-connectedness and localizability (see \cite[Definition~3.16 (ii)]{KS18}), every point of a KS-Lorentzian length space is the midpoint of a timelike curve. In particular, its chronological boundary is empty and $X$ coincides with its chronological convex hull, i.e.~$X=I(X)  \coloneqq I^+(X) \cap I^-(X)$.
A KS-Lorentzian length space is globally hyperbolic if and only if it is causal and all its causal diamonds are compact; see \cite[Corollary~3.8]{Mi23}. If $(X,\ell,\rho)$ is a globally hyperbolic KS-Lorentzian length space with Lorentzian distance $d$ and causal relation $J$, then $(X,d)$ is a BMS-Lorentzian length space with $J_d=J$. Indeed, every globally hyperbolic KS-Lorentzian length space is geodesic and has a finite, continuous Lorentzian distance; see \cite[Theorem 3.28 and 3.30]{KS18}. Condition (i) in Subsection~\ref{sub:BMS_spaces} is immediate. Moreover, by \cite[Proposition 3.13]{ACS20}, $X$ also satisfies  condition (iii). As all causal diamonds in $X$ are compact and $X$ is Hausdorff, also condition (ii) holds for the metric topology of $X$; see also \cite[Theorem~4.8]{BMS25}. The reverse triangle inequality for $\ell$ shows that $J \subset J_d$. As $X$ has an empty chronological boundary and is timelike path-connected, for $(x,y) \in J_d$ we can choose a sequence of points $x_j \ll x$ converging to $x$ as $j\to\infty$. Then $x_j \in I^-(y) \subset J^-(y)$ by definition of $J_d$, so $x \leq y$ as $J^-(y)$ is closed by causal simplicity. Moreover, by the discussion in Subsection~\ref{sub:BMS_spaces} the KS-Lorentzian length space is second-countable (or equivalently separable) if and only if the corresponding BMS-Lorentzian metric space is countably generated.

By the discussion in Subsection~\ref{sub:BM_spaces}, we also see that if $(X,\ell,\rho)$ is a globally hyperbolic KS-Lorentzian length space, then $(X,\ell)$ with the topology induced by $\rho$ 
is a globally hyperbolic BM-length metric spacetime.

Other examples of globally hyperbolic KS-Lorentzian length spaces are provided by Lorentzian cones over proper geodesic metric spaces as in \cite[Propositions~4.8 and 4.10]{AGKS23}. In the following subsection we discuss a specific class of those.

\subsection{Conical Minkowski spacetimes} \label{sub:conical_Minkowski}

Here we discuss special subsets of the Minkowski space $(\Lo^{n+1},\langle \cdot, \cdot \rangle)$, which can be regarded as Lorentzian cones with a (singular) warped product Friedmann–Lema\^{i}tre–Robertson–Walker metric (cf.~\cite{AGKS23}), as examples in the context of the above spacetime notions. This class of examples comprises simple cosmological models such as the \emph{Milne universe} and allows for an explicit description of the space of (strong) Cauchy sets; see Lemma~\ref{lem:cauchy_graph}.

Given a timelike vector $v_0 \in \Lo^{n+1}$, the (upper) hyperboloid in $\Lo^{n+1}$ is given by
\[
			\HH^n \coloneqq  \{v \in \Lo^{n+1} \mid \langle v,v\rangle = -1, \langle v,v_0\rangle < 0\} \subset I^+(O) \, .
\] 
The Minkowski inner product of $\Lo^{n+1}$ restricts to a complete Riemannian metric on $\HH^n$ of constant negative sectional curvature $-1$, i.e.~$\HH^n$ is isometric to the $n$-dimensional hyperbolic space. We write $\pi: I^+(O) \rightarrow \HH^n$ for the radial projection and $d_{\HH}$ for the induced metric on~$\HH^n$.

In \cite{LP25} a \textit{domain} in $\HH^n$ is defined as a non-empty connected $n$-dimensional Lipschitz submanifold of $\HH^n$ with boundary. For simplicity, we always assume in addition that a domain is smooth and closed here. By a \textit{conical Minkowski spacetime} we here mean a cone $M=\R_{\geq 0} \Omega$ for some domain $\Omega \subset \HH^n$. In \cite{LP25}, in contrast, only  the \emph{regular part} of $M_{\mathrm{r}}=\R_{> 0} \Omega$ is considered. Note that $M$ is convex if and only if $\Omega$ is convex. We further note that $\Lo^{n+1}$ induces the structure of a time-oriented Lorentzian manifold with boundary on $M_{\mathrm{r}}$. 

The lengths of causal curves contained in a conical Minkowski spacetime $M$ induce an extended time separation function $\ell=\ell_M$ and a Lorentzian distance $d=d_M$ on $M$ as in the case of a Lorentzian manifold discussed before. With this Lorentzian distance the regular part of a conical Minkowski spacetime $M$ can be represented as a Lorentzian warped product $\R_{> 0} \times_{f(t)=t^2}\Omega$; see \cite[Section~2]{AGKS23}. If $u,v\in M$ satisfy $\ell(u,v)\geq 0$ and the straight segment between $u$ and $v$ is contained in $M$, then $\ell(u,v
)=|u-v|$. In the non-convex case, it holds that 
\[
   d(u,v)^2=  |v|^2 +|u|^2 - 2|u||v| \cosh(d_{\Omega}(\pi(u),\pi(v))) \, ,
\]
for $u,v\in M$ with $\ell(u,v) \geq 0$ (see \cite[Eq.~(1)]{AGKS23}), where $d_{\Omega}$ is the length metric on $\Omega \subset \HH^n$ induced by $d_{\HH}$; cf.~\cite[Chapter~2]{BBI01} and \cite[Subsection~2.4]{LP25}.

A conical Minkowski spacetime is a geodesic KS-Lorentzian pre-length space with a finite, continuous Lorentzian distance; see \cite[Proposition~2.2, Corollary~2.4]{AGKS23}.  However, only its regular part is a \emph{globally hyperbolic} KS-Lorentzian \emph{length} space (see \cite[Theorem 4.8, Proposition 4.10]{AGKS23}). On the one hand, this is because the chronological past of the origin $O$ is empty, which violates the localizability condition in \cite[Definition 3.16, (ii)]{KS18}, and, on the other hand, causal diamonds $J(O,x)$, $x\in M$, are not compact in general, e.g.~for $\Omega=\HH^n$. Similarly, only the regular part of a conical Minkowski spacetime is a length metric spacetime \`a la Braun--McCann as the origin is not the midpoint of a timelike curve. By continuity of the Lorentzian distance, the maximal causal relation of a conical Minkowski spacetime $M$ is closed. Moreover, it coincides with the standard causal relation. This implies that all causal diamonds of a conical Minkowski spacetime over a compact domain with respect to its maximal causal relation are compact. The distinguishing property (iii) of Lorentzian metric spaces follows from the fact that the regular part is a KS-Lorentzian length space. In total, we see that a 
conical Minkowski spacetime over a compact domain is a geodesic Lorentzian metric space  with non-empty chronological boundary. In particular, it is not countably generated in the sense of Bykov--Minguzzi--Suhr, which is a common assumption in their work \cite{BMS25}; see also~\cite[Theorem 21]{Mi26}. Nevertheless, all assumptions of our extension of Minguzzi's Cauchy time function construction (see Theorem~\ref{thm:cauchy_time_functions_exist}) are satisfied.

\section{Inextendible curves}\label{sec:inext}

Recall from Subsection~\ref{sub:ms_spaces} that all causal and timelike curves in a Lorentzian pre-length space are assumed to be future-directed. In the following sections all statements about Lorentzian pre-length spaces also hold in the setting of KS-Lorentzian pre-length spaces with 
locally Lipschitz continuous curves instead of continuous curves, unless the KS-setting is not explicitly treated separately.

Our main results crucially depend on the notion of an inextendible curve. In the following, we introduce different versions of this notion and prove some preliminary results for them in the setting of a Lorentzian pre-length space $X$. While similar discussions can be found in \cite{BG25} and \cite{Mi26}, we give a unified definition in a more general setting and provide some alternative perspectives. To further motivate our discussion here, we point out that the definitions of inextendibility in \cite[Definition~3.10]{KS18}  and \cite[Definition~2.8]{BG25} allow for rather pathological, undesirable inextendible causal curves; see Remarks on Definition~\ref{dfn:inext_def}, (d). This is also one reason why existing characterizations of inextendible causal curves in KS-Lorentzian pre-length spaces turn out to be contradictory without further assumptions; see Remarks on Definition~\ref{dfn:inext_def}, (e). Moreover, the following definition is more general in the sense that we also define inextendible causal curves on (half-)closed intervals and introduce the notion of t-inextendibility for timelike curves. While the former is necessary for considering the completeness notions covered by Theorem~\ref{thm:complete} for spacetimes with nonempty boundary, the latter turns out to be useful for proving the existence of inextendible timelike curves;  see Lemma~\ref{lem:inex_extens} and Corollary~\ref{cor:inex_timelike}. 

\begin{dfn}\label{dfn:inext_def}
For a Lorentzian pre-length space $X$, a causal (timelike) curve $\gamma : I \rightarrow X$ defined on an interval $I \subset \R$ is called \emph{future (t-)extendible}, if there exists an order preserving homeomorphism $\varphi:I \rightarrow \hat I$ onto some interval $\hat I \subset \R$, an interval $\tilde I \subset \R$ containing $\hat I$ and some $t\in \tilde I$ with $t>s$ for all $s\in \hat I$, and a causal (timelike) curve $\tilde \gamma : \tilde I \rightarrow X$ that restricts to $\gamma\circ \varphi^{-1}$ on $\hat I$.

Otherwise, we call $\gamma$ \emph{future (t-)inextendible}. The notion \emph{past (t-)inextendible} is defined analogously. The curve $\gamma$ is called \emph{(t-)inextendible} if it is both future and past (t-)inextendible. 
\end{dfn}

This definition is clearly parametrization independent, but the reparametrization $\varphi$ is unhandy in practice. Instead, in the following, we will mostly work with the subsequent equivalent reformulation, which, in the KS-setting, is based on the fact that every locally Lipschitz continuous path in a metric space can be parametrized by arclength; see, for instance, \cite[Proposition~2.5.9]{BBI01}.

\addtocounter{thm}{-1}
\begin{refdef}  For a Lorentzian pre-length space $X$, a causal (timelike) curve $\gamma : I \rightarrow X$ defined on an interval $I \subset \R$, and parametrized by arclength in the KS-setting, is future (t-)extendible, if there exists  a larger interval $\tilde I \subset  [-\infty,\infty]$, resp. $\tilde I \subset  (-\infty,\infty)$ in the KS-setting, a $t\in \tilde I$ with $t>s$ for all $s\in  I$, and a causal (timelike) curve $\tilde \gamma : \tilde I \rightarrow X$ that restricts to $\gamma$ on $ I$. An analogous past statement holds. 
\end{refdef}
\begin{proof} The statement about continuous extendibility is clear. If, in a KS-Lorentzian pre-length space $X$, a causal curve $\gamma:I\rightarrow X$  parametrized by arclength is future extendible in the sense of the definition, then $\gamma$ restricted to $I\cap[0,\infty)$ has finite length since length is invariant under reparametrization and locally Lipschitz continuous curves on finite compact intervals are Lipschitz continuous. Moreover, the locally Lipschitz continuous causal  future extension $\hat \gamma$ of the reparametrization $\tilde \gamma$ of $\gamma$ provided by the definition gives rise to a locally Lipschitz continuous causal  extension of $\gamma$. The converse is obvious.
\end{proof}

In general, the two notions of inextendibility defined above differ in a KS-Lorentzian pre-length space; see Remark on Definition~\ref{dfn:inext_def}, (d). To avoid this subtlety, we consider KS-Lorentzian pre-length spaces $(X,\rho,\ll,\leq,d)$, which are \emph{$\rho$-compatible}, that is, for every $x\in X$ there is a neighborhood $U$ of $x$ and some $C>0$ such that every causal (timelike) curve contained in $U$ has $\rho$-arclength at most $C$; see \cite[Definition~3.13]{KS18}. We also refer to this property as \emph{metric-compatibility}.

\addtocounter{thm}{-1}
\begin{remsdfn}
(a) In a metric-compatible KS-Lorentzian pre-length space $X$, a causal curve $\gamma:[a,b)\rightarrow X$ is future (t-)inextendible in the locally Lipschitz continuous sense if and only if it is future (t-)inextendible in the continuous sense. Indeed, if $\gamma$ is future (t-)extendible by a continuous causal (timelike) curve $\tilde \gamma$ and without loss of generality assumed to be parametrized by $\rho$-arclength, then there is a neighborhood $U$ of $\tilde \gamma(b)$, some $C>0$, and some $t_0 \in [a,b)$ such that $\gamma([t_0,b)) \subset U$ and hence such that the length of  $\gamma$ restricted to $[t_0,b)$ is bounded by $C$. Since also the length of $\gamma$ on the compact finite interval $[a,t_0]$ is finite, we see that $b<\infty$ and thus that $\tilde \gamma_{|[a,b]}$ is a locally Lipschitz continuous causal (timelike) extension of $\gamma$. Analogous statements hold if $\gamma$ is defined on $(a,b)$ and if $\gamma$ is past (t-)inextendible.

(b) If a metric-compatible KS-Lorentzian pre-length space is in addition
causally (timelike) path-connected (or, more generally, if it has causal (timelike) extensions; see Definition \ref{dfn:extensions}), then for any interval $I$ a causal curve $\gamma:I\rightarrow X$ is future (t-)inextendible in the locally Lipschitz continuous sense if and only if it is future (t-)inextendible in the continuous sense. Analogous past statements hold. Note that, in particular, every KS-Lorentzian length space is metric-compatible as well as causally and timelike path-connected by definition.

(c) Under additional assumptions an alternative characterization of (t-)inextendibility can be given in terms of endpoints; cf.~\cite[Definition~11]{Mi26} and Lemma~\ref{lem:no_bubbling_(t)-inext} below.

(d) Without the reparametrization or the fixed parametrization by arclength in the KS-setting, the locally Lipschitz continuous causal curve $\gamma:[0,1) \rightarrow \R^{1+1}$ in Minkowski space $\R^{1+1}$ defined by $\gamma(t)= (1-\sqrt{1-t},0)$ would not be future extendible as a locally Lipschitz continuous causal curve (i.e.~in the sense of \cite[Definition~3.10]{KS18}, \cite[Definition~3.5]{Ta25}, and \cite[Definition~2.8]{BG25}). Allowing inextendible causal curves of this type would lead to an undesirable behavior of the notions of Condition A and Cauchy sets. For instance, Minkowski space would neither satisfy Condition A nor contain Cauchy sets. 

(e) The metric-compatibility assumption is necessary for the equivalence stated in (a) as the following example shows. We consider $X=[0,\infty]$ as a causal space with causal relation $\leq$ and timelike relation $<$. With respect to the order topology on $X$ and the Euclidean topology on $\R^2$ the map
\[
    c:X \rightarrow \R^2\,, \quad t\mapsto \begin{cases}
    \left(\frac{\cos t}{1+t},\frac{\sin t}{1+t} \right)  & \text{if } t\in [0,\infty)\,, \\
    0 & \text{else}\,, 
    \end{cases}   
\]
is a homeomorphism onto its image. Endowed with the metric $\rho$ obtained by restricting the Euclidean metric to $c(X)$ and identifying the latter with $X$, the pair $(X,\rho)$ is a compact (in particular proper) metric space. The map
\[
    d:X \times X \rightarrow [0,\infty)\,, \quad (t,s)\mapsto \begin{cases}
    \arctan(s)- \arctan(t) & \text{if } t\leq s\,, \\
    0 & \text{else}\,, 
    \end{cases}
\]
where $\arctan(\infty)\coloneqq \pi/2$, defines a (finite, continuous) Lorentzian distance on $X$ which turns $(X,\rho,<,\leq,d)$ into a KS-Lorentzian pre-length space. As a KS-Lorentzian pre-length space, $X$ is (globally hyperbolic in the sense of Subsection~\ref{sub:ms_spaces}, and, in particular,) locally causally closed in the sense of \cite[Definition 3.4]{KS18}. Since the curve $c_{|[0,\infty)}$ has infinite length, there exists a (future-directed) causal curve $\gamma:[0,\infty)\rightarrow X$ parametrized by arclength with $\lim_{t\rightarrow \infty} \gamma(t)=\infty \in  X$, which, according to \cite[Lemma~3.12]{KS18}  and \cite[Lemma~2.9]{BG25}, should be both future extendible and future inextendible. In fact, $\gamma$ is future extendible as a continuous causal curve and future inextendible as a locally Lipschitz continuous causal curve.
\end{remsdfn}

A causal curve on an \emph{open} interval in a Lorentzian metric space is inextendible in the sense of
\cite[Definition 11]{Mi26} if it has neither \emph{future} nor \emph{past endpoints}, which are defined as follows.

\begin{dfn}
A causal curve $\gamma:I \rightarrow X$ is said to have a \emph{future endpoint} if $\lim_{t \nearrow\sup I} \gamma(t)$ exists. We say that an \emph{endpoint is attained} if $\sup I \in I$. \emph{Past endpoints} and their attainability are defined analogously.
\end{dfn}

Clearly, having an endpoint is a necessary condition for attaining an endpoint. The next lemma states assumptions under which a future (past) endpoint provides a natural future (past) extension for a given causal curve on an open interval. In particular, it shows that for a first-countable Lorentzian metric space the notion of inextendibility in Definition~\ref{dfn:inext_def} and in \cite{Mi26} coincide for curves on open intervals; see also Lemma \ref{lem:no_bubbling_(t)-inext}.

\begin{lem}\label{lem:endpoint}
Suppose that $X$ is a first-countable, causally simple (metric-compatible KS-)Lorentz\-ian pre-length space. If a causal curve $\gamma:I \rightarrow X$ (parametrized by $\rho$-arclength in the KS-setting) has a future endpoint, then ($\sup I < \infty$ and) $\tilde \gamma:I \cup \{\sup I\} \rightarrow X$ defined by $\tilde{\gamma}(t)=\gamma(t)$ for $t\in I$ and $\tilde \gamma(\sup I)= \lim_{t\nearrow 
\sup I} \gamma(t)$ is a causal curve that restricts to $\gamma$ on $I$. Moreover, if $\gamma$ is timelike, then so is $\tilde \gamma$. An analogous past statement holds.
\end{lem}
\begin{proof}Since continuity and sequential continuity are equivalent in first-countable spaces, the curve $\tilde \gamma$ is continuous. Since $X$ is causally simple, $\tilde \gamma$ is causal. If $\gamma$ is timelike, then so is $\tilde \gamma$ by the push-up principle from Subsection~\ref{sub:ms_spaces}. In the KS-setting, Remark (a) on Definition~\ref{dfn:inext_def} now shows that $\sup I < \infty$ and that $\tilde \gamma$ is locally Lipschitz continuous. The past statement follows analogously.
\end{proof}

The existence of endpoints can be characterized in terms of the notion of imprisonment (see Lemma~\ref{lem:inextendible_charac}), which we describe next.

\begin{dfn}
A causal curve $\gamma:I \rightarrow X$ is called \emph{future} resp.~\emph{past
imprisoned} if there exists a compact set $K\subset X$ and some $a\in I$ such that $\gamma(t) \in K$ for every
$t> a$ resp.~$t<a$.
\end{dfn}

The equivalence in Lemma~\ref{lem:inextendible_charac} relies on the next lemma on the existence and uniqueness of ``causal limits'', which will be applied frequently in the following. The argument in its proof can also be found in the proof of \cite[Proposition~18]{Mi26}.

\begin{lem}\label{lem:unique_monot_limit}
    Let $X$ be a first-countable, causally simple Lorentzian pre-length space. Let $\alpha: \Sigma \rightarrow X$ be a function for some non-empty $\Sigma\subset \R$. 
    Assume the following two conditions:
    \begin{enumerate}
        \item $\alpha$ is causal, that is, $\alpha(t)\leq \alpha(s)$ for all $t\leq s$. 
        \item Every sequence $(\alpha(s_j))_{j\in \N}$ with $(s_j)_{j\in\N}\subset \Sigma$ and $s_j\nearrow s^*\coloneqq \sup \Sigma\in (-\infty,\infty]$ (resp.~$s_j\searrow s_*\coloneqq \inf \Sigma\in [-\infty,\infty)$) has an accumulation point.
    \end{enumerate} Then 
    all sequences $(\alpha(s_j))_{j\in \N}$ in (ii) have the same limit, which coincides with $\alpha(s^*)$ if $s^* \in \Sigma$ (resp.~$\alpha(s_*)$ if $s_* \in \Sigma$).    In particular, if the interval $(a,s^*)$ for some $a<s^{*}$ (resp.~$(s_*,a)$ for some $a>s_{*}$) is contained in $\Sigma$, then $\lim_{s\nearrow s^*} \alpha(s)$ (resp.~$\lim_{s\searrow s_*} \alpha(s)$) exists, and it coincides with $\alpha(s^*)$ if $s^* \in \Sigma$ (resp.~$\alpha(s_*)$ if $s_* \in \Sigma$).
\end{lem}

In our  applications, the set $\Sigma$ will be $\N$ or an interval.

\begin{proof} We only consider the case $s_j\nearrow s^*$. The other case is analogous. Assume by contradiction that there are sequences $(\alpha(s_j))_{j\in \N}$ and $(\alpha(t_j))_{j\in \N}$ with $s_j,t_j\nearrow s^*$ as $j\to \infty$ and $x\neq y \in X$ such that $\alpha(s_j)\to x$, $\alpha(t_j)\to y$ for $j\to\infty$. Choose a subsequence $(\alpha(t_{m_j}))_{j\in \N}$ with $$\alpha(s_j)\leq \alpha(t_{m_k})$$ for all $k\geq j$. This is possible as $\alpha$ is causal. Since $X$ is causally simple, $J^{+}(\alpha(s_j))$ is closed, so $$\alpha(s_j)\leq \lim_{k\to \infty}\alpha(t_{m_k})=y\,.$$ As this relation holds for all $j\in \N$ and $J^{-}(y)$ is closed (again by causal simplicity of $X$), we deduce $$x=\lim_{j\to\infty}\alpha(s_j)\leq y\,.$$
By symmetry of the argument, this also gives $y\leq x$, and thus causality of $X$ implies $x=y$, a contradiction.
\end{proof}

As a first application of this lemma, we establish the following equivalence between imprisonment and the existence of an endpoint.

\begin{lem}\label{lem:inextendible_charac}
Suppose that $X$ is first-countable, causally simple Lorentzian pre-length space. A causal curve $\gamma:I \rightarrow X$ is future imprisoned if and only if it has a future endpoint. An analogous past statement holds. 
\end{lem}

\begin{proof} If $s\coloneqq \sup I \in I$, then $\gamma$ is future imprisoned and has a future endpoint. We can hence assume that $s \notin I$.

 Suppose that $\gamma$ has a future endpoint. Then $\gamma$ has a proper continuous causal future extension $\tilde \gamma$ by Lemma~\ref{lem:endpoint}. For any $a\in I$, the curve $\gamma$ does not leave the compact set $\tilde \gamma ([a,\sup I])$ in future time, so that $\gamma$ is future imprisoned.

Conversely, suppose that $\gamma$ is future imprisoned in $K$. As $\gamma$ is causal and $K$ is compact (and thus sequentially compact), $\gamma:I\to X$ satisfies assumptions (i) and (ii) of Lemma~\ref{lem:unique_monot_limit}. Therefore, the limit $\lim_{s\nearrow\sup I}\gamma(s)$ exists, which means that $\gamma$ has a future endpoint.

The statement for the past case holds mutatis mutandis.
\end{proof}

Our next characterization of (t-)inextendible curves in the sense of Definition~\ref{dfn:inext_def} will be formulated under the following additional assumption.
 
 \begin{dfn}\label{dfn:extensions}
     We say that a Lorentzian pre-length space $X$ has \emph{causal (timelike) extensions} if for every $x \notin \partial_{ca}^+ X$ ($x \notin \partial_{ch}^+ X$) resp.~$x\notin\partial_{ca}^- X$ ($x \notin \partial_{ch}^- X$) there is a nonconstant causal (timelike) curve starting resp.~ending at $x$.
 \end{dfn}
 Note that a causally resp.~timelike path-connected Lorentzian pre-length space has causal resp.~timelike extensions. For instance, every KS-Lorentzian length space and every BM-length metric spacetime has both timelike and causal extensions. While every BMS-Lorentzian length space is causally path-connected, the next example shows that it does not need to have timelike extensions.

\begin{exl}

Let $X$ be a closed sector in $\Lo^{1+1}$ with its standard Lorentzian distance $d$ determined by a  null ray ending at some $p\in \Lo^{1+1}$ and a timelike  ray starting at $p$; see Figure~\ref{fig:placeholder}. Equip $X$ with the subspace topology and its ``induced Lorentzian distance'' defined as
\[
         d_L(x,y) \coloneqq \sup  \left\{L(\gamma) \mid \gamma : [a,b] \rightarrow X \, \text{causal curve with }\gamma(a)=x, \gamma(b)=y \right\} \leq d(x,y) 
\] if $x$ and $y$ can be connected by a causal curve in $X$ and as $d_L(x,y)=0$ otherwise. This distance coincides with $d$ if the straight segment between causally related $x$ and $y$ lies in $X$. In case it does not, the reverse triangle inequality shows that $d_L$ is realized by the concatenation of the straight segment from $x$ to $p$ and the straight segment from $p$ to $y$; see Figure~\ref{fig:placeholder}. In particular, the supremum in the definition of $d_L$ is attained. 

 Associating a 
causal and timelike relation to $(X,d_L)$ as described at the beginning of Section~\ref{sec:lorentzian_spaces} (in the context of Lorentzian manifolds) turns $X$ into a Lorentzian pre-length space. The space $X$ is geodesic, and the conditions (i)-(iii) in Subsection~\ref{sub:BMS_spaces} can easily be verified to conclude that $X$ is a BMS-Lorentzian length space. However, $X$ does not have a timelike extension starting at any point of the null boundary ray that ends at $p$ (except for $p$). Note that in this example $\partial_{ch} X=\emptyset$.

\begin{figure}
    \centering

\begin{tikzpicture}[scale=0.9]
  \begin{scope}[rotate=45]
    \def\s{1.2}
     
\draw[fill=gray!30, draw=none] (-1/2*\s,1/2*\s) -- (3/2*\s,-3/2*\s) -- (4*\s,\s)  -- (3*\s,2*\s) -- (2*\s,\s) -- (0, \s) -- cycle;
    
 \draw[line width=1pt] (0,\s) --  (2*\s,\s) -- (3*\s,2*\s);
   
    \fill (2*\s,\s) circle (3pt);  
    \node at (1.95*\s,1.35*\s) 
    {$p$};

     \fill (0.3*\s,0.7*\s) circle (3pt);  
    \node at (0.1*\s,0.5*\s)
    {$x$};

    \node at (3/2*\s,-1*\s)
    {$X$};

    \draw[line width=1pt] (0.3*\s,0.7*\s) --  (2*\s,\s);

    \draw[line width=1pt] (3.0*\s,1.6*\s) --  (2*\s,\s);

    \fill (3.0*\s,1.6*\s) circle (3pt);  
    \node at (3.2*\s,1.4*\s) 
    {$y$};
  \end{scope}
\end{tikzpicture}
    \caption{A geodesic between causally related points $x,y\in X$ that cannot be connected by a straight segment in $X$ has to pass $p$.}
    \label{fig:placeholder}
\end{figure}
\end{exl}

Finally, we can state the following characterization of (t-)inextendible curves.

\begin{lem}\label{lem:no_bubbling_(t)-inext}
Let $X$ be a first-countable, causally simple (metric-compatible KS-)Lorentzian pre-length space and $\gamma:I \rightarrow X$  a causal resp.~timelike curve. Then the following statements hold.
\begin{enumerate}
\item If $\sup I \notin I$, then $\gamma:I \rightarrow X$ is future inextendible resp.~t-inextendible if and only if $\gamma$ does not have a future endpoint.
\item If $X$ has causal resp.~timelike extensions and $\sup I \in I$, then $\gamma:I \rightarrow X$ is future inextendible resp.~t-inextendible if and only if $\gamma$ attains a future endpoint in the future causal resp.~chronological boundary of $X$.
\end{enumerate}
In particular, if
$X$ has timelike extensions and no bubbling boundary, then a timelike curve is future inextendible if and only if it is future t-inextendible. Analogous past statements hold.
\end{lem} 

\begin{proof} We prove the characterization of future inextendible causal curves. The characterization of future t-inextendible timelike curves and the respective past statements follow analogously. By the metric-compatibility assumption it suffices to prove the claim in the metric-free setting, the claim in the KS-setting follows then as well as discussed in Remarks on Definition \ref{dfn:inext_def}, (a) and (b). In this case (i) follows from Lemma~\ref{lem:endpoint}. If $\gamma$ attains a future endpoint in the causal boundary, then it is clearly inextendible. Conversely, if $\gamma$ is future inextendible and attains a future endpoint, then, assuming the existence of causal extensions, this endpoint has to be contained in the causal boundary, as otherwise $\gamma$ could be extended.\end{proof}

Now we provide sufficient conditions for the existence of t-inextendible timelike curves.

\begin{lem}\label{lem:inex_extens}
Suppose that $X$ is either a Lorentzian pre-length space, which has a finite, continuous Lorentzian distance $d$, or a KS-Lorentzian pre-length space. Then every timelike curve $\gamma: I \rightarrow X$ is, up to reparametrization, the restriction of a t-inextendible timelike curve. 
\end{lem}
\begin{proof}
For any $t_0\in I$ it suffices to show that the restriction of $\gamma$ to $[t_0,\infty) \cap I$ is, up to reparametrization, the restriction of a future t-inextendible timelike curve. The general statement then follows analogously. After reparametrization (by arclength in the KS-setting), we can assume that $s\coloneqq \sup I < \infty$, that $ \sup I \in I$ (so that $\gamma$ attains a future endpoint $x_0 \in X$), and that $\gamma$ is future t-extendible.

We now first work in the continuous setting and deal with the KS-setting afterwards. We consider the (non-empty) set of timelike curves $\hat \gamma : \hat I \rightarrow X$ satisfying $s = \min \hat I \in \R$, $\hat \gamma(s)=x_0$, and $d(x_0,\hat \gamma(t))=t-s$ for all $t \in \hat I$ with its natural order. Note that, since $d$ is finite and continuous, every timelike curve starting at $x_0$ can be brought to this form after a suitable reparametrization. By Zorn's lemma there is a maximal such extension $\gamma_*$, which is future t-inextendible by maximality. The concatenation of $\gamma$ and $\gamma_*$ is a future t-inextendible extension of $\gamma$. 

In a KS-Lorentzian pre-length space, we consider the (non-empty) set of timelike curves $\hat \gamma : \hat I \rightarrow X$ parametrized by $\rho$-arclength with  $s = \min \hat I \in \R$ and $\hat \gamma(s)=x_0$ with its natural order. Recall that every timelike curve starting at $x_0$ can be parametrized by $\rho$-arclength; see, for instance, \cite[Proposition~2.5.9]{BBI01}. Again by Zorn's lemma there is a maximal such curve $\gamma_*:I_* \rightarrow X$. If $\sup I_* < \infty$, then $\gamma_*$ is future $t$-inextendible by maximality. If $\sup I_*=\infty$, then $\gamma_*$ is future (t-)inextendible by the Lemma on Definition~\ref{dfn:inext_def}.
\end{proof}

\addtocounter{thm}{-1}
\begin{remslem}\label{rem:rem_on_inext}
(a) The same argument shows that every causal curve in a KS-Lorentzian pre-length space $X$ is, up to reparametrization, the restriction of an inextendible causal curve.
In particular, through every point of a Lorentzian pre-length space there runs an inextendible causal curve. If the assumptions of Lemma~\ref{lem:no_bubbling_(t)-inext} are satisfied and the causal boundary of $X$ is empty, then these inextendible curves are defined on open intervals; cf.~\cite[Proposition 2.22]{BG25}.

(b) If the assumptions of Lemma~\ref{lem:no_bubbling_(t)-inext} are satisfied, the causal resp.~chronological boundary of $X$ is empty, and $X$ is in addition proper (e.g.~a \emph{proper} globally hyperbolic KS-Lorentzian length space), then the proof of \cite[Lemma 3.12]{KS18} together with Lemma~\ref{lem:no_bubbling_(t)-inext} shows that the case $\sup I_* < \infty$ in the proof of the lemma above cannot occur. Hence, in this case we see that through every point there exists an inextendible causal resp.~timelike curve defined on an open interval and with infinite length. In Corollary~\ref{cor:inex_timelike} below, we establish the existence of inextendible timelike curves without assuming properness of the metric.
\end{remslem}

In Lorentzian metric spaces inextendible causal curves can be constructed via suitable time functions; see \cite[Proposition~20]{Mi26}. The following corollary, in particular, provides sufficient conditions that yield inextendible timelike curves in a more direct way in our setting of interest.

\begin{cor}\label{cor:inex_timelike} Suppose that $X$ is a first-countable, causally simple Lorentzian pre-length space with empty bubbling boundary, which either is a metric-compatible KS-Lorentzian pre-length space or has a finite, continuous Lorentzian distance. If $X$ has timelike extensions, then every timelike curve $\gamma:I \rightarrow X$ is, up to reparametrization, the restriction of an inextendible timelike curve. In particular, if $X$ has timelike extensions, then for every point $x_0\in X$ there exists an inextendible timelike curve that runs through $x_0$. The converse of the latter statement holds if $X$ has causal extensions.
\end{cor}

Recall that if not only the bubbling boundary, but also the chronological boundary of $X$ is empty, then the inextendible timelike curves provided by the corollary are defined on open intervals by Lemma~\ref{lem:no_bubbling_(t)-inext}.

\begin{proof} 
First, assume that $X$ has timelike extensions. After reparametrization the timelike curve $\gamma$ can be extended to a t-inextendible timelike curve $\gamma_*$ by Lemma~\ref{lem:inex_extens}. Since the bubbling boundary is empty, the future causal and chronological boundary coincide. Therefore, the curve $\gamma_*$ is also inextendible by Lemma~\ref{lem:no_bubbling_(t)-inext}. In particular, for $x_0 \in X$ we can extend the constant curve $\gamma$ at $x_0$ to an inextendible timelike curve through $x_0$ in this way. 

Conversely, assume that $X$ has causal extensions and that $\gamma:I\to X$ is an inextendible timelike curve running through a given $x_0\in X$ with $0\in I$ and $\gamma(0)=x_0$. If $x_0\notin \partial^+ X$ ($x_0\notin \partial^- X$), then there is a nonconstant causal curve starting (ending) at $x_0$, in which case the inextendible curve $\gamma$ cannot end (start) at $x_0$.  
Hence, restricting $\gamma$ to $I\cap[0,\infty)$ (to $I\cap(-\infty,0]$) yields the desired timelike extension.
\end{proof}

The example $X=\{(0,0),(0,1)\} \subset \Lo^{1+1}$ with the restricted structures shows that assuming the existence of causal extensions is necessary for the converse in the lemma.

As a consequence, we observe that through every point of a globally hyperbolic KS-Lorentzian length space there runs an inextendible timelike curve defined on an open interval. In particular, these inextendible timelike curves are also inextendible in the sense of \cite{BG25}; cf.~\cite[Corollary 2.23]{BG25}.

\section{Completeness conditions }\label{sec:completeness}
Based on the previous section, we can generalize statements on completeness properties of spacetimes due to Beem \cite{Be76} and Takahashi \cite{Ta25} and prove further 
variants thereof. We first introduce the 
basic definitions that are needed in the Theorems~\ref{thm:metric_properties} and~\ref{thm:complete}. The following definition goes back to Busemann \cite{Bu67}.

\begin{dfn}\label{dfn:finitely_compact}
A Lorentzian pre-length space $X$ is called \emph{finitely compact} if for every $B>0$, $x,y\in X$, and sequence of points $(x_j)_{j\in \N}$ in $X$ with either $x \ll y \leq x_j$ and $d(x,x_j) \leq B$ for all  $j\in \N$ or $x_j \leq y \ll x$ and $d(x_j,x) \leq B$ for all $j\in \N$, there is an accumulation point of $(x_j)_{j\in \N}$ in $X$.
\end{dfn}

For instance, Minkowski space $\Lo^{n+1}$ and convex conical Minkowski spacetimes are easily seen to be finitely compact. For an example due to Geroch of a finitely compact, globally hyperbolic Lorentzian manifold that is not timelike geodesically complete, we refer to \cite{Ge68}; cf.~\cite[Remark~6.13]{BEE96}.

If $X$ is second-countable, then $X$ is finitely compact if and only if for all $x \ll y$ in $X$ and all $r>0$ the sets  $
\{ z \in J^+(y) \mid d(x,z) \leq r\}
$ and $
\{ z \in J^-(x) \mid d(z,y) \leq r\}
$ are compact; cf.~\cite[Lemma 3.2]{Ta25}. However, unlike in \cite{Ta25}, this characterization is not used in the following.

\begin{dfn}\label{dfn:cauchy_complete}
A Lorentzian pre-length space $X$ is called \emph{timelike Cauchy complete} if for any sequence $(x_j)_{j\in \N}$ in $X$ such that either  

(i) $x_j \ll x_{j+1}$ for all $j\in \N$ and $\sup_{k}d(x_j,x_{j+k}) \to 0$ as $j\to \infty$ or

(ii) $x_{j+1} \ll x_j$ for all $j\in \N$ and $\sup_kd(x_{j+k},x_j)   \to 0$ as $j\to \infty$,

\noindent
the sequence $(x_j)_{j\in \N}$ converges in $X$. 
\end{dfn}

Beem \cite[Lemma~3]{Be76} proved that finitely compact, globally hyperbolic Lorentzian manifolds are timelike Cauchy complete, which was extended to globally hyperbolic KS-Lorentzian length spaces by Takahashi in \cite[Theorem~3.4]{Ta25}. Our next proposition provides a further generalization with a streamlined argument.  

\begin{prp}\label{prp:finitely_compact_weakly_complete} 
Let $X$ be a first-countable, causally simple Lorentzian pre-length space.
If $X$ is finitely compact, then it is timelike Cauchy complete.
\end{prp}

\begin{proof} Let $(x_j)_{j\in \N}$ be a sequence in $X$ as in Definition~\ref{dfn:cauchy_complete} with $x_j \ll x_{j+1}$ and $j_0\in \N$ large enough such that $\sup_{k}d(x_j,x_{j+k})$ is bounded for all $j\geq j_0$. The other case is analogous. Set $x=x_{j_0}$, $y=x_{j_0+1}$, and $B=\sup_k d(x_{j_0},x_{j_0+k})$. By finite compactness as stated in Definition~\ref{dfn:finitely_compact}, the sequence $(x_{j_0+j})_{j\in \N}$ satisfies condition (ii) in Lemma~\ref{lem:unique_monot_limit}. Since $(x_{j_0+j})_{j\in \N}$ is causal, also condition (i) is met. We can therefore conclude that $(x_{j})_{j\in \N}$ converges. 
\end{proof}

Takahashi \cite[Theorem 3.8]{Ta25} showed that a timelike Cauchy complete, globally hyperbolic KS-Lorentzian length space satisfies the following condition introduced by Beem in the setting of Lorentzian manifolds; see \cite{Be76}. Here a \emph{geodesic} 
is defined as a causal curve that is locally maximal; see \cite[Definition 3.6]{Ta25}.

\begin{dfn}
A Lorentzian pre-length space is said to satisfy \emph{Condition A} if for any $x,y \in X$ with $x \ll y$  resp.~$y \ll x$ and any future resp.~past inextendible causal geodesic $\gamma:I \rightarrow X$ with $0\in I$ and $\gamma(0)=y$, it holds that $d(x,\gamma(t)) \rightarrow \infty$ as $t \nearrow \sup I$ or $\sup I \in I$ resp.~$d(\gamma(t),x) \rightarrow \infty$ as $t\searrow \inf I$ or $\inf I\in I$.
\end{dfn}

Note that an inextendible causal geodesic is in particular inextendible as a geodesic and that in a Lorentzian manifold also the converse holds. Indeed, in this case a geodesic is up to reparametrization the same as a solution of the geodesic equation (see e.g.~\cite[Proposition~5.34]{On83}), so a geodesic with an endpoint is always extendible as a geodesic. Also note that the case $\sup I \in I$ resp.~$\inf I \in I$ does not occur in a KS-Lorentzian length space 
since its chronological boundary is empty by \cite[Definition 3.16, (ii)]{KS18}.

For us, the following completeness conditions turn out to
be more useful than Condition~A. 

\begin{dfn}
We call a future inextendible timelike curve $\gamma$ in a Lorentzian pre-length space $X$ defined on an interval $I\subset \R$ \emph{future Cauchy complete} provided the following condition is satisfied: if there is some $t_0 \in I$ for which $d(\gamma(t_0),\gamma(\cdot))_{|[t_0,\infty) \cap I}$ is bounded, then $\sup I$ is contained in $I$. An analogous definition is made for past inextendible timelike curves. We call $\gamma$ \emph{Cauchy complete} if it is future and past Cauchy complete. We say that $X$ satisfies \emph{Condition B} if every future inextendible timelike curve is future Cauchy complete and every past inextendible timelike curve is past Cauchy complete.
\end{dfn}

In what follows a (future resp.~past) Cauchy complete timelike curve is by definition (future resp. past) inextendible. 

Note that a future inextendible causal curve $\gamma: [0,a) \rightarrow X$ of infinite Lorentzian length is future Cauchy complete. However, infinite Lorentzian length is not necessary: If $x(t)$ is a unit speed curve in $\HH^n$ and $r(t) = \sqrt{\sinh(2(t+1))}$, then $\gamma: [0,\infty) \rightarrow \Lo^{n+1}$, defined by $\gamma(t)=r(t)x(t)$, is of finite length, future inextendible, and future Cauchy complete. 

In fact, the Minkowski space $\Lo^{n+1}$ is timelike Cauchy complete, and we now show that timelike Cauchy completeness implies Condition B.

\begin{prp}\label{prp:tcc_propB} Let $X$ be a first-countable, causally simple (metric-compatible KS-)Lorentzian pre-length space. If $X$ is timelike Cauchy complete, then it satisfies Condition B.
\end{prp}

\begin{proof} Let $\gamma: I \rightarrow X$ be a future inextendible timelike curve. We prove that $\gamma$ is future Cauchy complete. The proof of the corresponding past statement is analogous. We can assume that $s\coloneqq\sup I \notin I$, since otherwise $\gamma$ is future Cauchy complete. 
We can further assume that there is some $t_0 \in I$ for which $d(t)\coloneqq d(\gamma(t_0),\gamma(t))$, $t \in [t_0,s)$, is bounded. By Lemma~\ref{lem:endpoint} the claim follows by contradiction, provided that we can show that $\gamma$ has a future endpoint. To prove the existence of the limit $x=\lim_{t\nearrow s } \gamma(t)$, we observe that for any sequence $(t_j)_{j\in \N}$ with $t_j \nearrow s$ as $j\to\infty$ the sequence defined by $x_j=\gamma(t_j)$, $j\in \N$, satisfies $x_j\ll x_{j+1}$ as $\gamma$ is timelike. Observe further that by the reverse triangle inequality $$0< d(x_j,x_{j+k}) \leq d(t_{j+k})-d(t_j)\to 0$$ as $j\to\infty$, uniformly in $k\in \N$, where we used that $(d(t_j))_{j\in \N}$ is a Cauchy sequence by monotonicity and boundedness of $d(t)$. Therefore, the existence of the limit $x=\lim_{j\to \infty} \gamma(t_j)$ is guaranteed by Definition~\ref{dfn:cauchy_complete}. Since $\gamma$ is timelike, it meets the conditions of Lemma~\ref{lem:unique_monot_limit}, so we conclude $x=\lim_{t \nearrow s} \gamma(t)$. 
 Hence, the claim follows. 
\end{proof}

Now we introduce the following strengthening of Condition A.

\begin{dfn}
A Lorentzian pre-length space satisfies \emph{Condition A$^*$} if it satisfies Condition A not only for all future inextendible causal geodesics and all past inextendible causal geodesics, but for all future inextendible and all past inextendible causal curves.
\end{dfn}

We can say more about the relation between the other previously considered conditions and Condition A$^*$. Recall that KS-Lorentzian length spaces are by definition timelike path-connected and as a metric space first-countable.

\begin{prp}\label{prp:tcc_propAB} Let $X$ either be a globally hyperbolic KS-Lorentzian length space, a separable, timelike path-connected BMS-Lorentzian metric space without chronological boundary, or a separable, timelike path-connected, globally hyperbolic BM-length metric spacetime. If $X$ satisfies Condition B, then it also satisfies Condition A$^*$ and hence Condition A.
\end{prp}

Note that in the KS-setting the following proof goes through no matter whether Conditions B, A$^*$, and A are interpreted in terms of continuous or locally Lipschitz continuous causal curves.

\begin{proof}
Let $x,y \in X$ with $x\ll y $ and let $\gamma: I \rightarrow X$ be a future inextendible causal curve with $\min I =0$ and $\gamma(0)=y$ without loss of generality. We can assume that $c \coloneqq \sup I \notin I$ so that $I=[0,c)$. By the push-up principle from Subsection~\ref{sub:ms_spaces}, we have $\gamma(t) \in I^+(x)$ for all $t \in [0,c)
$ and by continuity of the Lorentzian distance $I^+(x)$ is open. Since $X$ is timelike path-connected and $I^-(\gamma(t))$ is non-empty and open for all $t \in [0,c)$, it follows that $I^+(x) \cap I^-(\gamma(t))$ is non-empty for all $t \in [0,c)$. Therefore, given a sequence $(t_j)_{j\in \N}$ with $0=t_0 < t_j \nearrow c$ as $j\to\infty$, by timelike path-connectedness, we can successively construct by concatenation a timelike curve $\tilde \gamma:[0, c) \rightarrow X$ with $\tilde \gamma(0) =x$, $\tilde \gamma(t_j) \in I^+(\tilde \gamma(t_{j-1})) \cap I^-(\gamma(t_j))$, and $\rho(\tilde \gamma(t_j),\gamma(t_j)) \leq 1/j$ for all $j\in \N$, where $\rho$ denotes a metric on $X$ that induces the topology of $X$; see Subsections \ref{sub:ks_spaces} and~\ref{sub:BMS_spaces}, respectively. We claim that $\tilde \gamma$ is future inextendible. Assume by contradiction that it is not.
Then $\tilde \gamma$ has a future endpoint as the extension is continuous, 
and also $(\gamma(t_j))$ converges to it by the preceding paragraph. Recall from Subsection \ref{sub:BM_spaces} that every globally hyperbolic BM-length metric spacetime can be regarded as a BMS-Lorentzian length space without chronological boundary. Moreover, in the BMS-setting separability and an empty chronological boundary imply that $X$ is countably generated, in particular, that $X$ is first-countable and $I(X)=X$; see Subsection \ref{sub:BMS_spaces}. In the KS-setting we can regard $X$ as a Lorentzian metric space with $I(X)=X$ as well as discussed in Subsection~\ref{sub:ks_spaces}. Hence, in all cases \cite[Proposition~18]{Mi26} implies that also $\gamma$ has a future endpoint, in contradiction to Lemma~\ref{lem:no_bubbling_(t)-inext}. Now, as $X$ satisfies Condition B and $\tilde \gamma:[0,c) \rightarrow X$ is inextendible, the function $t \mapsto d(x,\tilde \gamma(t))$ is unbounded. As $d(x,\tilde \gamma(t_j)) \leq d(x,\gamma(t_j))$ by the reverse triangle inequality, also $d(x, \gamma(t))$ is unbounded in $t$. Since $t\mapsto d(x,\gamma(t))$ is moreover increasing by the reverse triangle inequality, we have $d(x,\gamma(t)) \rightarrow \infty$ as $t \rightarrow c$. With an analogous past argument, it follows that $X$ satisfies Condition A$^*$.
\end{proof} 

\begin{rem} \label{rem:boundary_version_prop_B} The statement of the proposition also holds for a metrizable, locally compact, timelike path-connected BMS-Lorentzian metric space. Indeed, in this case the proof of \cite[Proposition~18]{Mi26} still shows that in the above proof $\gamma$ has a future endpoint, in contradiction to Lemma~\ref{lem:no_bubbling_(t)-inext}, so that the claim follows as before.\end{rem}

Finally, we record the following conclusion in the smooth manifold setting.

\begin{cor}\label{cor:completeness_smooth_case}
For a smooth, globally hyperbolic Lorentzian manifold, the notions finite compactness, timelike Cauchy completeness, Condition B, Condition A$^*$, and Condition A (both in terms of inextendible continuous causal curves and in terms of inextendible locally Lipschitz continuous causal curves) are all equivalent.
\end{cor}

\begin{proof}
We have seen that the strength of the listed conditions is decreasing in the stated order with respect to both notions of inextendible causal curves. However, Condition A is defined in terms of geodesics, which are smooth and whose inextendibility is independent of whether it is defined in terms of geodesics, continuous causal curves, or locally Lipschitz continuous causal curves. Since in a smooth, globally hyperbolic Lorentzian manifold Condition A implies finite compactness by \cite[Theorem~5]{Be76}, the claim follows.
\end{proof}

\subsection{Proof of Theorem~\ref{thm:complete}}
Propositions~\ref{prp:finitely_compact_weakly_complete} and~\ref{prp:tcc_propB} prove Theorem~\ref{thm:complete}, (i) and (ii), respectively. The final part of Theorem~\ref{thm:complete} follows from Proposition~\ref{prp:tcc_propAB}.

\section{A Hausdorff-type metric on the space of Cauchy sets} \label{sec:hausdorff_cauchy}

In this section we prove Theorem~\ref{thm:metric_properties} and analogous statements in the setting of Lorentzian pre-length spaces. Before we turn to Cauchy sets, we establish metric space properties on larger subspaces of subsets.

\subsection{Metric space properties on subsets of Lorentzian pre-length spaces}\label{sub:metric_space_properties}

Let $(X,\ell)$ be a causal  Lorentzian pre-length space with associated Lorentzian distance $d$.
Following an idea by Bahn and Ehrlich \cite{BE99}, for $x,y \in X$ we symmetrize the Lorentzian distance $d_J(x,y)=d(x,y)+d(y,x) = \max \{d(x,y),d(y,x) \}$ and  set 
\begin{equation*}
 d_J(A,B)\coloneqq \sup_{x \in A,y \in B} d_J(x,y) 
\end{equation*} for $A,B\subset X$. Moreover, we write $d_J(A,x) \coloneqq d_J(x,A) \coloneqq d_J(\{x\},A)$ for $x\in X$ and $A\subset X$. The map $d_J$ is clearly symmetric and non-negative but in general lacks all other properties of a metric; cf.~\cite[Section~5]{BE99}. As observed by Bahn and Ehrlich in the special case of Minkowski space, the map $d_J$ has some formal resemblance with the Hausdorff metric and is better behaved when restricted to certain subsets of $X$. Namely, writing
 \begin{equation}
     \label{eq:nghb}N(A,\varepsilon) \coloneqq \{ x \in X \mid d_J(x,A)  \leq \varepsilon \}
 \end{equation}
for $\varepsilon>0$, for $A,B \subset X$ it holds that $d_J(A,B) \leq \varepsilon$ if and only if $A \subset N(B,\varepsilon)$ and $B \subset N(A,\varepsilon)$; cf.~\cite[Lemma~5.1]{BE99}. 
To obtain better properties, we first show a variant of \cite[Lemma~5.1]{BE99} based on the following notions.

\begin{dfn} A subset $A$ of a Lorentzian pre-length space is called \emph{achronal} if $I^+(A) \cap A = \emptyset$ and \emph{chronologically observing} if $X\backslash A \subset I^+(A)\cup I^-(A)$.
\end{dfn}

Note that a subset of a timelike path-connected Lorentzian pre-length space is achronal if and only if no timelike curve intersects it more than once.

\begin{lem}[Definiteness]\label{lem:separation}
For a causal Lorentzian pre-length space $X$ and $A,B\subset X$, the following statements hold.
\begin{enumerate}
\item[(i)]  If $A\backslash B\subset I^+(B)\cup I^-(B)$ and $B\backslash A \subset I^+(A)\cup I^-(A)$, then $d_J(A,B)=0$ implies that $A=B$. 
\item[(ii)]  If $A$ is achronal, then $d_J(A,A)=0$.
\end{enumerate}
\end{lem}

Note that the condition in (i) holds in particular if $A$ and $B$ are chronologically observing.
\begin{proof} (i) Suppose that $A$ and $B$ are distinct. Then there is without loss of generality some $x \in A\backslash B$ with $x\in I^+(B)$ by assumption. This means that there is some $y\in B$ with $x \in I^+(y)$ and thus $d_J(A,B) \geq d_J(x,y) >0$. 

(ii) Let $x,y\in A$. By definition we have $d_J(x,y)=0$ if $y \notin I^+(x) \cup I^-(x)$. Since $I^+(A) \cap A = \emptyset$ also implies $I^-(A) \cap A = \emptyset$, the claim follows.
\end{proof} Unlike Bahn and Ehrlich, we now also provide conditions under which $d_J$ satisfies the triangle inequality.

\begin{dfn}\label{dfn:weak timelike-path-obser} A subset $A$ of a causal Lorentzian pre-length space $X$ is called \emph{weakly timelike intercepting} if for any pair of points $x \ll z$ in $X$ there is a point $y\in A$ which satisfies one of the following conditions
 \begin{enumerate}
\item $z\leq y$
\item $x \leq y \leq z$ and $d(x,z)  = d(x,y)+d(y,z)$
\item $y \leq x$
 \end{enumerate}
\end{dfn} Such subsets have the following property relevant to the triangle inequality.
\begin{lem}\label{lem:pre_time_path_observ} Let $A$ be a weakly timelike intercepting subset of a causal Lorentzian pre-length space $X$. Then for every pair of points $x \ll z$ in $X$  there is a point $y\in A$ with 
 \begin{equation}\label{eq:triangle_pf}
   d_J(x,z) \leq d_J(x,y) + d_J(y,z)\,. 
 \end{equation}
 \end{lem}
 \begin{proof} Let $y\in A$ be the point provided by Definition~\ref{dfn:weak timelike-path-obser}. We can assume that we are in case (i) or (iii). If $y \leq x$, we have  $d(y,x)\geq 0$ and
\[
    d_J(x,z)=d(x,z) \leq d(y,z)-d(y,x) \leq d(y,z)+d(y,x)\leq d_J(y,z)+d_J(x,y)
\] by the reverse triangle inequality. The case $z \leq y$  is analogous.
\end{proof}

We can now establish the triangle inequality for $d_J$ on weakly timelike intercepting sets. For more concrete sufficient assumptions under which the following lemma holds, we refer to Subsection~\ref{sub:metric_on_cauchy_sets}.

\begin{lem}[Triangle inequality] \label{lem:tri_ineq}
 Every triple of weakly timelike intercepting subsets $A$, $B$, $C$ of a causal Lorentzian pre-length space $X$ satisfies
\[
   d_J(A,C) \leq d_J(A,B) + d_J(B,C)\,.
\]
 \end{lem}
 
 \begin{proof} 
 We first assume that $d_J(A,C)< \infty$. Let $\varepsilon>0$. 
 There are $x \in A$ and $z \in C$ such that $d_J(x,z)+\varepsilon>d_J(A,C)$ and $d(x,z)>0$ without loss of generality. By Lemma~\ref{lem:pre_time_path_observ} there is a point $y \in B$ such that \eqref{eq:triangle_pf} holds. Hence, we find that
\begin{equation*} d_J(A,C)-\varepsilon<d_J(x,z)  
    \leq d_J(A,B) + d_J(B,C)\,,
\end{equation*} which completes the proof in this case by sending $\varepsilon\to 0$.

If $d_J(A,C)= \infty$, we can find by definition of $d_J$ points $x\in A$ and $z\in C$ such that $d(x,z)>0$ is arbitrarily large and $d(x,z) \leq d_J(A,B) + d_J(B,C)$ by Lemma~\ref{lem:pre_time_path_observ}. Hence, also the right hand side of the triangle inequality is infinite.
\end{proof}

\subsection{The space of Cauchy sets }\label{subsec:Cauchy_sets}

First, we define different versions of Cauchy sets and discuss their existence and properties. The following definition is meaningful if there are sufficiently many inextendible curves of the respective type. Otherwise, every subset can be a Cauchy set; see Examples~\ref{exl:null_space} and~\ref{exl:non_geod_complete_examples}. For instance, 
there is an inextendible causal curve (defined on an open interval) through every point of a KS-Lorentzian length space with a proper metric by \cite[Corollary~2.23]{BG25} and of a causally path-connected, countably generated Lorentzian metric space by \cite[Proposition~20]{Mi26}. Sufficient conditions on a Lorentzian pre-length space for the existence of an inextendible timelike curve through every point are provided in Corollary~\ref{cor:inex_timelike}. For instance, there is an inextendible timelike curve (defined on an open interval) through every point of a globally hyperbolic KS-Lorentzian length space. In a timelike geodesically complete Lorentzian manifold, there even exists a Cauchy complete timelike curve through every point.

\begin{dfn}\label{dfn:cauchy_set}
A subset of a Lorentzian pre-length space $X$ is called a \emph{Cauchy set} if it is intersected exactly once by every inextendible timelike curve. We call it a 
\emph{strong Cauchy set} if it is intersected exactly once by every inextendible causal curve. We call it a \emph{weak Cauchy set} if it is intersected exactly once by every Cauchy complete timelike curve. We denote the set of (strong resp.~weak) Cauchy sets in $X$ by ($\Cauchy^s_X$ resp.~$\Cauchy^w_X$) $\Cauchy_X$. The Lorentzian pre-length space $X$ is called \emph{spatially compact} if it contains a Cauchy set and all its Cauchy sets are compact.
\end{dfn}

Every Cauchy set in a globally hyperbolic Lorentzian manifold is a Lipschitz hypersurface (see e.g.~\cite[Theorem~2.147]{Mi19}), which is why it is usually called a \emph{Cauchy hypersurface} in this case. In a Lorentzian manifold $M$ neither the precise regularity of a causal resp.~timelike curve nor the precise sense of ``timelike'' (see Subsection~\ref{sub:ms_spaces}) does make a difference for the definition of a (strong) Cauchy set, as explained in more detail in the next remark. 

\begin{rem}\label{rem:timelike_variation} First, recall from Remarks on Definition~\ref{dfn:inext_def} that the notion of inextendibility is unambiguous in case of a globally hyperbolic Lorentzian manifold $M$. Suppose that a subset $S$ of $M$ is intersected exactly once by every inextendible locally Lipschitz continuous a.e.~timelike (causal) curve $\gamma$ (i.e.~$\gamma'(t)$ timelike (causal) for almost all~$t$). If an inextendible continuous timelike (causal) curve intersects $S$ twice, say in $x,y\in M$ with $x \ll y$ ($x\leq y$), then one can also construct an inextendible locally Lipschitz continuous a.e.~timelike (causal) curve intersecting $S$ twice, namely as a concatenation of a past inextendible locally Lipschitz continuous timelike (causal) curve ending at $x$, a smooth timelike (causal) geodesic from $x$ to $y$, and a future inextendible locally Lipschitz continuous timelike (causal) curve starting at $y$ by an application of a Lipschitz version of Corollary~\ref{cor:inex_timelike} and the fact that $M$ is regular and geodesic. Similarly, if an inextendible continuous timelike (causal) curve does not intersect $S$, then one can construct an inextendible locally Lipschitz continuous a.e.~timelike (causal) curve that does not intersect $S$; see the proofs of \cite[Lemma~14.29]{On83} and Proposition~\ref{prp:tcc_propAB}. In view of \cite[Lemma~3.1 and Corollary~3.2]{LP25}, the same discussion applies to conical Minkowski spacetimes. 
\end{rem}

In a Lorentzian manifold, a Cauchy hypersurface is intersected by every inextendible causal curve, and every causal curve is the restriction of an inextendible causal curve; see \cite[Lemma 14.29]{On83} and \cite[Proposition 2.22]{BG25}, for instance. Hence, in this case a Cauchy hypersurface $S$ is strong if and only if it is \emph{acausal}, i.e.~intersected at most once by every causal curve. In Minkowski space it is easy to construct Cauchy hypersurfaces that are not acausal.

As mentioned in the introduction, a Lorentzian manifold has a Cauchy hypersurface if and only if it is globally hyperbolic, and in this case all Cauchy hypersurfaces are homeomorphic; see \cite[Theorem 11]{Ge70}. In fact, every globally hyperbolic Lorentzian manifold has a smooth spacelike, and hence acausal, Cauchy hypersurface, as proved in \cite{BS03}. 

The characterization by Geroch has been generalized to second-countable, globally hyperbolic KS-Lorentzian length spaces with a proper metric structure in \cite[Theorem~1.3]{BG25}. The existence of a strong Cauchy set in a countably generated Lorentzian metric space is shown in \cite[Theorem~21]{Mi26}.  In fact, in all these cases the authors establish the existence of a (strong) ``Cauchy time function'', that is, a time function whose level sets are all strong Cauchy sets; see Appendix~\ref{app:time_functions}. There we explain that the proof of Minguzzi's result also provides a strong Cauchy time function for every separable, first-countable Lorentzian metric space $X$ with compact causal diamonds, empty bubbling boundary, and empty spacelike boundary. In particular, the statement also applies to separable, globally hyperbolic BM-length metric spacetimes, to separable (or equivalently second-countable), globally hyperbolic KS-Lorentzian length spaces (without assuming properness of the metric structure), and to conical Minkowski spacetimes over a compact domain, not all of which are covered by Minguzzi's result; cf.~Subsection~\ref{sub:BM_spaces},~\ref{sub:ks_spaces}, and~\ref{sub:conical_Minkowski}. Recall in this respect that by Remarks on Definition~\ref{dfn:inext_def}, (a) and (b), a causal curve in a KS-Lorentzian length space is inextendible as a continuous causal curve if and only if it is inextendible as a locally Lipschitz continuous curve. Hence, the strong Cauchy sets provided by this result in the setting of a globally hyperbolic KS-Lorentzian length space are also strong Cauchy sets in the KS-sense with locally Lipschitz continuous curves.

The following statement is an immediate consequence of the definition of Condition B.

\begin{cor}\label{cor:weak cauchy_is_cauchy}
In a first-countable, causally simple Lorentzian pre-length space $X$ that satisfies Condition B, every weak Cauchy set is a Cauchy set, i.e.~$\Cauchy^w_X=\Cauchy_X$.
\end{cor}

To complement the abstract notion of a Cauchy set with a more concrete example, we state a characterization of Cauchy sets in conical Minkowski spacetimes from \cite[Lemma 3.5]{LP25}; see Subsection~\ref{sub:conical_Minkowski} and Remark~\ref{rem:timelike_variation}. In particular, the lemma provides an example of a spatially compact Lorentzian pre-length space.

\begin{lem}\label{lem:cauchy_graph}
 For a compact smooth domain $\Omega \subset \HH^n$, the set of (strong) Cauchy sets of the regular part $M_{\mathrm{r}}$ of a conical Minkowski spacetime $M= \R_{\geq 0} \Omega$ is given by the set of graphs $S_f$ of
locally Lipschitz functions $f:\Omega \rightarrow \R_+$ for which $\ln(f)$ is (strictly) $1$-Lipschitz continuous with respect to the intrinsic metric of $\Omega$.
\end{lem}

Since every point in $M$ lies on an inextendible timelike ray starting at $O$, the only Cauchy set of $M$ that is not contained in $M_{\mathrm{r}}$ is $\{O\}$.

Finally, we record the following basic observation that is needed later.

\begin{lem}\label{lem:prop_cauchy_set} Let $X$ be a timelike path-connected Lorentzian pre-length space such that there exists an inextendible timelike curve through every point in $X$ and let $S$ be a Cauchy set in $X$. Then $X=I^+(S) \cup S \cup I^-(S)$ and $J^{\pm}(S)=S\cup I^{\pm}(S)$ are disjoint unions. In particular, $S$ is closed.\end{lem} 

\begin{proof}
    Note that the above union exhausts $X$ since through every point there exists an inextendible timelike curve intersecting $S$. If $S\cap I^\pm(S)\neq \emptyset$, then by timelike path-connectedness and the existence of inextendible timelike curves through every point we can construct an inextendible timelike curve that crosses $S$ twice, a contradiction to $S$ being a Cauchy set. 
    
    If $x\in J^+(S)\cap I^-(S)$ (and analogously if $x\in J^-(S)\cap I^+(S)$), then there are $y,y'\in S$ such that $y\leq x\ll y'$, but by the push-up principle $y\ll y'$, again a contradiction to $S$ being a Cauchy set. This in particular implies $I^+(S)\cap I^-(S)=\emptyset$. 
    
    As $I^\pm(S)$ are open, $S$ is closed.
\end{proof}

\subsection{The metric space of Cauchy sets} \label{sub:metric_on_cauchy_sets}
In Subsection~\ref{sub:metric_space_properties} we have shown that the map $d_J$ restricts to an extended metric on the space of subsets of a Lorentzian pre-length space that
are achronal, chronologically observing, and weakly timelike intercepting. The following lemma provides sufficient conditions for the first two properties in the case of Cauchy sets.

\begin{lem}\label{lem:cauchy_is_achronal_and_chro_obs} (i) Let $X$ be a causally path-connected Lorentzian pre-length space. If there exists an inextendible causal (timelike) curve through every point in $X$, then every strong Cauchy set in $X$ is achronal (and chronologically observing).

(ii) Let $X$ be a timelike path-connected Lorentzian pre-length space. If there exists a (Cauchy complete) inextendible timelike  curve through every point in $X$, then every (weak) Cauchy set in $X$ is achronal and chronologically observing.
\end{lem}

\begin{proof}
We first prove achronality in (i) and (ii). Assume first that $S$ is a strong Cauchy set in $X$ that is not achronal. Then there are points $x \ll y$ in $S$ that can be connected by a causal curve $\gamma$. By assumption this curve can be extended to an inextendible causal curve $\hat \gamma$, in contradiction to the definition of a strong Cauchy set. Under the additional assumption(s), the curve $\hat \gamma$ can be constructed to be timelike (and Cauchy complete) so that the contradiction also follows if $S$ is only a (weak) Cauchy set in this case.

Now we turn to the chronologically observing property. For a (weak) Cauchy set $S$ in $X$ and a point $x\in X \backslash S$, there is by assumption a (Cauchy complete) inextendible timelike curve through $x$. By the defining property of a (weak) Cauchy set, it intersects $S$ in a point $y$. Since the curve is timelike, we have $x \in I^+(y)\cup I^-(y)\subset I^+(S)\cup I^-(S)$, so $S$ is chronologically observing. Since we do not use any of the additional assumptions, strong Cauchy sets are in particular chronologically observing.
\end{proof}

Before we turn to the last desired condition, we discuss examples of geodesic Lorentzian pre-length spaces for which there is no inextendible timelike curve through some point and whose space of strong Cauchy sets is a non-empty pseudometric space, i.e.~a space whose metric does not necessarily satisfy definiteness, but not a metric space.

 \begin{exl}\label{exl:null_space}
 (a) We consider $X=(0,1) \subset \R$ with its subspace topology, the standard relation $\leq$, and the trivial Lorentzian distance $d=0$ on $X$. Then $(X,\leq,d)$ defines a  Lorentzian pre-length space with an empty chronological relation. In particular, there are no nonconstant timelike curves. It is globally hyperbolic in the sense that it is causal, its causal relation is closed, and all its causal emeralds are compact. 
 The causal curve $\gamma:(0,1) \rightarrow X$, $t\mapsto t$, is  inextendible and maximal, and any other causal curve in $X$ is up to reparametrization a restriction of $\gamma$. In particular, $X$ is geodesic (and thus causally path-connected), and the space of strong Cauchy sets in $X$ is given by $\mathcal{C}^s_X=X$. All strong Cauchy sets in $X$ are achronal and weakly timelike intercepting but none of them is chronologically observing. The space $(\mathcal{C}^s_X,d_J=0)$ is a pseudometric space but not a metric space. Since there are no inextendible timelike curves in $X$, every subset of $X$ is a Cauchy set.
  
(b) In the previous construction the causal boundary is empty, but each point belongs to both the future and past chronological boundary, and hence to the bubbling boundary. We can remedy this pathology by considering instead the ``induced Lorentzian length metric'' on an open polygonal arc $X$ in $\Lo^{1+1}$ consisting of a timelike segment, followed by a null segment, followed by a timelike segment with the restricted causal structure.  In this example the chronological boundary is empty. Moreover, the set of strong Cauchy sets in $X$ still coincides with $X$, and $d_J$ is trivial on subsets of $\mathcal C_X^s$ in the null segment. There are still no inextendible timelike curves, and thus every subset of $X$ is a Cauchy set.
 \end{exl}

To guarantee that Cauchy sets are weakly timelike intercepting, we assume the space $X$ to be geodesic (which already implies causal path-connectedness).

\begin{prp} \label{prp:conditions_ext_metric} Let $X$ be a geodesic Lorentzian pre-length space.

(i) If through every point of $X$ there is an inextendible causal  (resp.~timelike) curve, then $(\Cauchy^s_X,d_J)$ is an extended pseudometric (resp.~metric) space.

(ii) Suppose in addition that $X$ is first-countable and
timelike path-connected and has a continuous Lorentzian distance and compact causal diamonds. If through every point of $X$ there is a (Cauchy complete) inextendible timelike curve, then also $(\Cauchy_X,d_J)$ (and $(\Cauchy^w_X,d_J)$) is an extended metric space.
\end{prp}

The proof shows that in (ii) we could assume regular instead of first-countable, timelike path-connected, and compact causal diamonds.

\begin{proof}
By Lemma~\ref{lem:tri_ineq} and Lemma~\ref{lem:cauchy_is_achronal_and_chro_obs},
we are left to verify in (i), resp.~in (ii), that a strong Cauchy set, resp.~a (weak) Cauchy set, $S$ in $X$ is weakly timelike intercepting. To this end, we consider $x,z \in X$ with $x \ll z$. As $X$ is geodesic, there is a maximal causal curve $\gamma:[0,1]\rightarrow X$ from $x$ to $z$, i.e.~$L(\gamma) = d(x,z)$.

To prove (i), we extend $\gamma$ using the assumption to an inextendible causal curve $\hat \gamma$, which has to intersect $S$ in some point, say in $y\in S$. If neither $z\leq y$ nor $y\leq x$ holds, then we have $x \leq y \leq z$ as $x,y,z$ lie on a causal curve.  In this case we write $\gamma$ as a concatenation of its part $\gamma^1$ from $x$ to $y$ and its part $\gamma^2$ from $y$ to $z$. Then
\[
  d(x,z) = L(\gamma) = L(\gamma^1) + L(\gamma^2) \leq 
  d(x,y) + d(y,z) \, ,
\]
so, by the reverse triangle inequality, $S$ is weakly timelike intercepting.

To deduce (ii), we now refine the above argument to show that under the additional assumptions also every (weak) Cauchy set is weakly timelike intercepting. As the Lorentzian distance is continuous and $d(x,z)>0$, we can choose for any $\varepsilon>0$ a subdivision $0=t_0<t_1<\ldots<t_n=1$ such that the restriction of $\gamma$ to $[t_i,t_{i+1}]$ has positive Lorentzian length at most $\varepsilon$ for all $i=0,\ldots,n-1$. For such $i$ let $\gamma_i:[t_i,t_{i+1}]\rightarrow X$ be a timelike curve with $\gamma_i(t_i)=\gamma(t_i)$ and $\gamma_i(t_{i+1})=\gamma(t_{i+1})$, which exists by timelike path-connectedness of $X$, and let $\gamma_{*}$ be the concatenation of the $\gamma_i$. The (Cauchy complete) inextendible timelike curve $\hat \gamma_*$ constructed from $\gamma_*$ as before intersects a (weak) Cauchy set $S$ in some point $y_\varepsilon \in J(x,z)$. By the argument above we can assume that $\gamma_*$ intersects $S$, say in $\gamma_j(t)$, $j\in \{0,\ldots,n-1\}$, $t\in [t_j,t_{j+1}]$. Let $\gamma^1$ be the concatenation of $\gamma_{|[0,t_j]}$ and ${\gamma_j}_{|[t_j,t]}$ and let $\gamma^2$ be the concatenation of ${\gamma_j}_{|[t,t_{j+1}]}$ and $\gamma_{|[t_{j+1},1]}$. Then
\[
  d(x,z) = L(\gamma) \leq L(\gamma^1) + L(\gamma^2) + \varepsilon \leq 
  d(x,y_\varepsilon) + d(y_\varepsilon,z) + \varepsilon\, ,
\]
By compactness of the causal diamond $J(x,z)$ and first countability of $X$, a subsequence of the $y_\varepsilon$ converges to some $y \in J(x,z)$ as $\varepsilon \rightarrow 0$ with $y \in S$ since $S$ is closed by Lemma~\ref{lem:prop_cauchy_set}. Continuity of the Lorentzian distance further implies that  $d(x,z) \leq d(x,y) + d(y,z)$ and the claim follows as before. 
\end{proof}

We obtain the following corollary of Proposition~\ref{prp:conditions_ext_metric} for some more specific notions of Lorentzian pre-length spaces. Regarding (iii) below, recall that a metric space is separable if and only if it is second-countable. 

\begin{cor}\label{cor:metric_corollary} Suppose that one of the following three conditions holds for a separable Lorentz\-ian pre-length space $X$.
\begin{enumerate}
    \item $X$ is a first-countable BMS-Lorentzian length space with compact causal diamonds, empty bubbling boundary, empty spacelike boundary, and timelike extensions.
    \item $X$ is a globally hyperbolic BM-length metric spacetime. 
    \item $X$ is a globally hyperbolic KS-Lorentzian length space.
\end{enumerate}
Then $(\Cauchy^s_X,d_J)$ is a non-empty extended metric space. If $X$ is in addition timelike path-connected (which it is in case (iii) by definition), then $(\Cauchy_X,d_J)$ is an extended metric space. If in addition through every point of $X$ there exists a Cauchy complete timelike curve, then $(\Cauchy^w_X,d_J)$ is an extended metric space. 
\end{cor}

In all the cases of $X$ above, $X$ is geodesic, causally simple and it has a finite, continuous Lorentzian distance, timelike extensions, and neither bubbling nor spacelike boundary; see Subsection~\ref{sub:ms_spaces},~\ref{sub:BMS_spaces},~\ref{sub:BM_spaces}, and~\ref{sub:ks_spaces}  for references for the respective properties. Moreover, $X$ is by definition  metric-compatible in case (iii). Hence, the existence of an inextendible timelike curve through every point is guaranteed by Corollary~\ref{cor:inex_timelike}, and Proposition~\ref{prp:conditions_ext_metric} is applicable. The fact that the respective space of strong Cauchy sets is non-empty follows from the discussion in Subsection~\ref{subsec:Cauchy_sets}; see also Appendix~\ref{app:time_functions}. Note that compactness of causal diamonds in (i) is required  for Proposition~\ref{prp:conditions_ext_metric} and to guarantee non-emptiness of $(\mathcal {C}^s_X, d_J)$ via Theorem~\ref{thm:cauchy_time_functions_exist}.

Note further that the assumptions in (i) are in particular satisfied for a countably generated BMS-Lorentzian length space with timelike extensions; cf.~\cite[Theorem~4]{Mi26}.
If we only assume the BMS-Lorentzian length space to be countably generated in (i), then $(\Cauchy^s_X,d_J)$ is still a (non-empty by \cite[Theorem 21]{Mi26}) extended pseudometric space, which follows from Proposition~\ref{prp:conditions_ext_metric}, (i), as there exists an inextendible causal curve through every point by \cite[Proposition~20]{Mi26}, and by  causal path-connectedness of $X$.

\begin{qst}
Is there an example of a countably generated BMS-Lorentzian length space for which $(\mathcal{C}^s_X, d_J)$ is not an extended metric space?
\end{qst}

With an additional compactness assumption, we actually obtain (finite) metric spaces.

\begin{cor}\label{cor:compact_finite} 
If in the situation considered in Corollary~\ref{cor:metric_corollary} all (strong resp.~weak)  Cauchy sets are compact, then the respective extended metric space is a metric space.
\end{cor}
\begin{proof}
In all the situations considered in  Corollary~\ref{cor:metric_corollary}, the Lorentzian distance is continuous. Therefore, finiteness of the metric $d_J$ follows from compactness of the Cauchy sets.
\end{proof}

Note that in a globally hyperbolic Lorentzian manifold all Cauchy hypersurfaces are homeomorphic; see \cite[Property 7]{Ge70}.

\subsection{Completeness properties}
In view of Lemma~\ref{lem:cauchy_graph}, we cannot expect that a space of strong Cauchy sets is complete. The following example shows that without completeness assumptions on a globally hyperbolic spacetime also the space of Cauchy sets $(\Cauchy_X,d_J)$ is in general not complete.

\begin{exl}\label{exl:non_geod_complete_examples} Given a splitting $\Lo^{1+1}=\R \times \R$ into a timelike subspace and a spacelike subspace, the Lorentzian manifold $M = (0,1)\times \R \subset  \Lo^{1+1}$ is globally hyperbolic and the time slice $S_t=\{t\} \times \R$ is a Cauchy hypersurface in $M$ for every $t\in (0,1)$ with $d_J(S_t,S_s)=|t-s|$. Hence, $(S_{1/j})_{j\in \N}$ is a Cauchy sequence in $(\Cauchy_M,d_J)$ without a limit in $(\Cauchy_M,d_J)$. Since $M$ is neither timelike geodesically complete nor timelike Cauchy complete, we see that these assumptions are necessary in Theorem~\ref{thm:metric_properties}, (i) resp.~(ii). In fact, in this example there are no Cauchy complete inextendible timelike geodesics, hence every subset is a weak Cauchy set, and so $(\mathcal{C}^w_M,d_J)$ is not even an extended (pseudo)metric space.
\end{exl}

However, under an additional completeness assumptions on the spacetime, we can guarantee completeness of the space of (weak) Cauchy hypersurfaces. Note that by Corollary~\ref{cor:inex_timelike} the assumptions of the following proposition imply that through every point of $X$ there runs an inextendible timelike curve and that such an inextendible timelike curve is timelike Cauchy complete if $X$ satisfies Condition B.

\begin{prp} \label{prp:conditions_compl_metric} Let $X$ be a first-countable, timelike path-connected, geodesic, causally simple (metric-compatible KS-)Lorentzian pre-length space with a finite, continuous Lorentzian distance, compact causal diamonds, and empty bubbling boundary. If there exists a Cauchy complete timelike curve through every point of $X$, then the extended metric space $(\mathcal C_X^w,d_J)$ is complete. In particular, if $X$ satisfies Condition B, 
then $(\Cauchy_X,d_J)=(\Cauchy^w_X,d_J)$ is complete. 
\end{prp}

Recall that Condition B follows from timelike Cauchy completeness or finite compactness (see Theorem~\ref{thm:complete}), and that under the assumptions of the proposition $(\mathcal{C}_X,d_J)$ resp.~$(\mathcal{C}^w_X,d_J)$ is an extended metric space by Proposition \ref{prp:conditions_ext_metric}, (ii).

\begin{proof}
The second statement is a consequence of the first statement and Corollary~\ref{cor:weak cauchy_is_cauchy}. For the first statement, we consider a Cauchy sequence $(S_j)_{j\in \N}$ in $(\Cauchy^w_X,d_J)$. Let $S$ denote the set of all points $x\in X$ such that for any neighborhood $U$ of $x$ in $X$ it holds that $U\cap S_j \neq \emptyset$ for infinitely many $j$. We claim that $d_J(S_j, S) \rightarrow 0$ and that $S$ is a weak Cauchy set.

For the first claim, we fix $\varepsilon >0$ and choose $N\in \N$ such that $d_J(S_j,S_k) < \varepsilon$ for all $j,k \geq N$. 
By construction, for any $x \in S$ there is a sequence $x_{l} \in S_{j_l}$ 
with $x_{l} \to x$ for $l \to \infty$. By assumption we have $d_J(x_{l},y) < \varepsilon$ for all $y\in S_k$ with $j_l,k \geq N$. Continuity of the Lorentzian distance implies that $d_J(x,y)\leq\varepsilon$ for all $x\in S$ and all $y\in S_k$ with $k \geq N$. Hence, $d_J(S,S_k) \leq \varepsilon$ for $k \geq N$, and thus the first claim follows. 

To prove that $S$ is a weak Cauchy set, we first show that $S$ is intersected by every Cauchy complete timelike curve. Let $\gamma :I \rightarrow X$ be such a curve. Let $t_j \in I$ be such that $\gamma(t_j) \in S_j$. After reparametrization we can assume that $t_1=0$ and $d_J(\gamma(0),\gamma(t))=|t|$. By Cauchy completeness of $\gamma$, the interval $I$ is closed in $\R$. Since $S_j$ is a Cauchy sequence, the sequence $d_J(S_1,S_j)$ is bounded. As  $|t_j| 
= d_J(\gamma(0), \gamma(t_j)) \leq d_J(S_1,S_j)$, also the sequence $(t_j)_{j\in \N}$ is bounded. Hence, a subsequence converges to some $t_*\in I$. By construction we have $\gamma(t_*) \in S$.

Finally, we show that every Cauchy complete timelike curve intersects $S$ at most once. Suppose that there is a Cauchy complete timelike curve $\gamma:I\rightarrow X$ that intersects $S$ in two distinct points $x \ll y$. We can suppose that $0\in I$ and $\gamma(0)=y$. By passing to a subsequence, we can assume that there is a sequence of points $x_{j} \in S_{j}$ with $x_{j} \rightarrow x$ as $j\to\infty$. Suppose that there is some neighborhood of $y$ that is not intersected by infinitely many $S_j$. Then there is a neighborhood $U$ of $y$ disjoint from all $S_j$. As $X$ is geodesic, it has causal extensions. Hence, it also has timelike extensions by Corollary~\ref{cor:inex_timelike}. We first suppose that $y$ is not contained in the chronological boundary of $X$. As $X$ has timelike extensions and $\gamma$ is inextendible, there is some $\varepsilon>0$ with $[-\varepsilon,\varepsilon]\subset I$ and $\gamma([-\varepsilon,\varepsilon]) \subset U$. Set $\delta \coloneqq \min \{d_J(y,\gamma(\varepsilon)),d_J(\gamma(-\varepsilon),y)\}$ and let $t_j \in I$ be such that $\gamma(t_j) \in S_j=S_j \backslash U$. By the reverse triangle inequality, we find that $d_J(S,S_j) \geq d_J(y,\gamma(t_j)) \geq \delta$ for all $j$, in contradiction to $d_J(S,S_j)\rightarrow 0$. The case $y\in \partial_{ch}^+X$ is treated analogously by noting that $\sup I =0$ in this case and considering curves on intervals $[-\varepsilon, 0]\subset I$. The case $y\in \partial_{ch}^-X$ cannot occur as $x \ll y$. 
 Hence, after passing to another subsequence, there are points $y_{j} \in S_{j}$ with  $y_{j} \rightarrow y$ as $j \rightarrow \infty$. By continuity of $d$, we can thus find some $j_0$ such that $x_{j_0} \ll y_{j_0}$. 
As $X$ is assumed to be timelike path-connected, the points $x_{j_0}$ and $ y_{j_0}$ can be connected by a timelike curve. By concatenation with a future and a past Cauchy complete inextendible timelike curve, we obtain a Cauchy complete timelike curve through $x_{j_0}$ and $ y_{j_0}$, a contradiction.
\end{proof}

The proposition immediately applies in the more specific settings of Corollary~\ref{cor:metric_corollary}. For instance, we see that the space of Cauchy sets in a globally hyperbolic KS-Lorentzian length space satisfying Condition B is complete.

\subsection{Geodesics in the space of Cauchy sets}\label{sub:geodesics} 

In this subsection we assume that 
$X$ is a first-countable, causally simple, timelike path-connected, geodesic (metric-compatible KS-)Lorentzian pre-length space with a finite, continuous Lorentzian distance, compact causal diamonds, no bubbling boundary, and the property that it satisfies Condition B; cf.~Proposition~\ref{prp:conditions_compl_metric}. Timelike Cauchy complete, globally hyperbolic Lorentzian manifolds, and more generally timelike Cauchy complete, globally hyperbolic KS-Lorentzian length spaces, are examples of such spaces satisfying all the conditions above. Under these assumptions, there is an inextendible timelike curve through every point of $X$ by Corollary~\ref{cor:inex_timelike}. Therefore, also the assumptions of Lemma~\ref
{lem:prop_cauchy_set} are satisfied, and $d_J$ defines a metric on the space of Cauchy sets in $X$ by Proposition~\ref{prp:conditions_ext_metric}, (ii). We collect these assumptions and refer to them as \textit{standing assumptions} on the Lorentzian pre-length space $X$ in the following. 

Our goal is to show that the space of Cauchy sets is a geodesic metric space in the above setting and to relate certain geodesic rays in it to Cauchy time functions of $X$. To this end, for a Cauchy set $S$ in $X$ and $t\in \R$, we define a \textit{$t$-parallel set} $S_t$ of $S$ in $X$ to be
\[
    S_t \coloneqq \{ x \in J^{\mathrm{sign}(t)}(S) \mid d_J(x,S) = |t| \} \cup \{x \in \partial^{\mathrm{sign}(t)}X \mid d_J(x,S) < |t|\}
\] for $t\neq 0$ and as $S_0\coloneqq S$ for $t=0$. The second part of this union guarantees that inextendible curves that cease to exist at the chronological boundary $\partial X$ of $X$ still hit $S_t$. In fact, under suitable assumptions, these parallel sets are again Cauchy sets.

\begin{lem}\label{lem:parallel_cauchy} Assume the standing assumptions on a Lorentzian pre-length space $X$ as given at the beginning of this subsection. Let $S$ be a Cauchy set in $X$. If $d_J(S,\cdot)$ is continuous, then the parallel sets $S_t$ are Cauchy sets for all $t\in \R$.
\end{lem}
\begin{proof} Suppose that $t>0$, the case $t<0$ being analogous, and let $\gamma:I\rightarrow X$ be an inextendible timelike curve in $X$.

We first show that $\gamma$ intersects $S_t$. Since $S$ is a Cauchy set, it is intersected by $\gamma$ in a unique point, say, without loss of generality, in $\gamma(0)=x\in S$. If $d_J(\gamma(s),S)\geq t$ for some $s\in I \cap[0,\infty)$, then $\gamma$ intersects $S_t$ by continuity. Hence, we can assume that $d_J(\gamma(s),S) < t$ for all $s\in I \cap[0,\infty)$. Then also $d_J(\gamma(0),\gamma(s)) < t$ for all $s\in I \cap[0,\infty)$. Since $X$ satisfies Condition B, $\gamma$ is Cauchy complete, which implies that $\gamma$ attains a future endpoint $\gamma(\sup I) \in \partial^+X$ with $\sup I \in I$ by inextendibility and $d(\gamma(0),\gamma(\sup I)) \leq t$. In particular, $\gamma$ intersects $S_t$.

 Suppose $\gamma$ intersects $S_t$ twice, say, in $y,z\in S_t$ with $y \ll z$. Then there is some $x\in S$ and some $\lambda \in (0,1)$ with $d(x,y) > t-\lambda d(y,z)>0$. Therefore, we have $x\ll y \ll z$ and 
\[
   d(x,z) \geq d(x,y) + d(y,z) > t
\]
by the reverse triangle inequality, a contradiction. Hence, $S_t$ is a Cauchy set. 
\end{proof}

In the context of closed cone structures, in particular Lorentzian manifolds, continuity of the symmetrized Lorentzian distance function $d_J(S,\cdot)$ is discussed in \cite{Mi20}. We record the following known conditions ensuring continuity of $d_J(S,\cdot)$; cf. \cite[Proposition~1.3]{Mi20}.

\begin{lem}\label{lem:continuity_distance_fct} For a first-countable,  causal Lorentzian pre-length space $X$ with a finite continuous Lorentzian distance and a Cauchy set $S\subset X$, the function $d_J(S,\cdot)$ on $X$ is lower semicontinuous. Moreover, it is continuous under any of the following additional assumptions.
\begin{enumerate}
    \item $S$ is compact.
    \item $X$ is a smooth, globally hyperbolic Lorentzian manifold and $S$ is acausal.
\end{enumerate}
\end{lem}

Since parts of the proof in \cite{Mi20} are omitted, we provide further details in Appendix~\ref{app:continuity_Lorentzian_distance} for the convenience of the reader.

When working in the space of acausal Cauchy hypersurfaces of a Lorentzian manifold, it can be useful to know that acausality is inherited by parallel sets. However, since this statement is not central in the following, we postpone its proof to Appendix~\ref{sub:acausality_of_Cauchy_sets}.

\begin{lem}\label{lem:acausal_parallel_Sets} If $S$ is an acausal Cauchy hypersurface in a smooth, globally hyperbolic Lorentzian manifold $M$ satisfying Condition B, then all parallel sets $S_t$ are acausal Cauchy hypersurfaces.
\end{lem}

To show that a complete extended metric space is geodesic, by
\cite[Theorem 2.4.16]{BBI01} (applied to a sufficiently large ball on which the metric is finite), it suffices to prove the existence of midpoints. In our setting and under suitable assumptions, we next describe a midpoint of two given Cauchy sets $S$ and $S'$, i.e.~we describe a third Cauchy set $S''$ satisfying
\[
    d_J(S,S'') = d_J(S',S'') = \frac 1 2 d_J(S,S') \, .
\]
In fact, we show the following slightly more general statement. 

\begin{lem}\label{lem:mid} Assume the standing assumptions on a Lorentzian pre-length space $X$ as given at the beginning of this subsection. Suppose $S,S' \in \Cauchy_X$ are Cauchy sets for which $d_J(S,\cdot)$ and $d_J(S',\cdot)$ are continuous. Then for every $s,t\geq 0$ with $s+t = d_J(S,S')$ the set
\[
    S''=S_-'' \coloneqq ( S_s \cap J^-(S_t')) \cup (S_t' \cap J^-(S_s) )
\]
is a Cauchy set in $X$ with $d_J(S,S'')=s$ and $d_J(S'',S')=t$. Moreover, if $X=M$ is  
a smooth, globally hyperbolic Lorentzian manifold and $S$, $S'$ are acausal, then also $S''$ is acausal.
\end{lem}

\begin{figure}
	\centering
    \def\svgwidth{0.6\textwidth}
		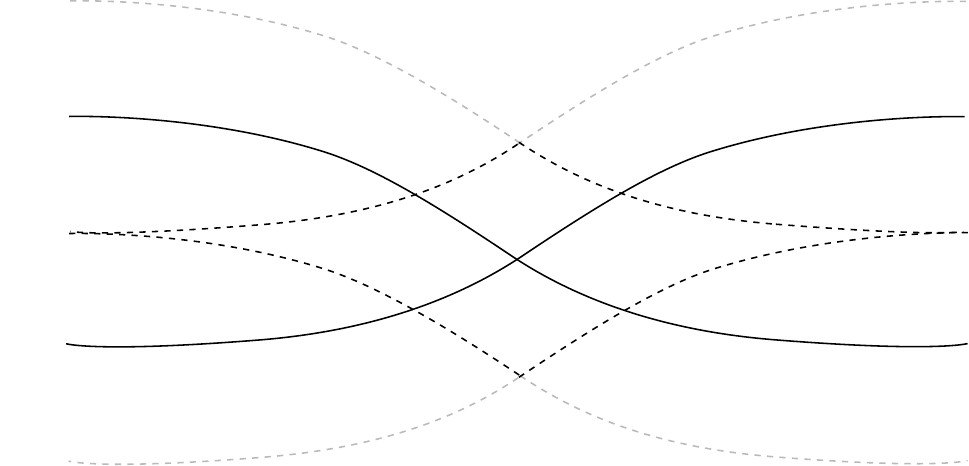
	\caption{Schematic picture of two Cauchy sets $S$ and $S'$ (solid lines) together with three midpoints (dashed lines) that satisfy $s=t=1/2$: $S_+''$ as defined after Lemma \ref{lem:mid}, $S_-''$ as defined in Lemma \ref{lem:mid}, and a third one between between  $S_+''$ and $S_-''$. The boundary of the neighborhoods $N(S,1/2)$ and $N(S',1/2)$ is indicated by faded dashed lines. 
}
	\label{fig:_triangle}
\end{figure}

Recall that a Cauchy hypersurface in a smooth, globally hyperbolic Lorentzian manifold is acausal if and only if it is strong; see~the discussion below Remark~\ref{rem:timelike_variation}. In general, the Cauchy set $S''$ is not unique. For instance, the set
\[
    S_+''\coloneqq( S_{-s} \cap J^+(S_{-t}')) \cup (S_{-t}' \cap J^+(S_{-s}) )\] also satisfies the conclusion of the lemma.

\begin{proof} First, we show that $S''$ is a Cauchy set. Let $\gamma$ be an inextendible timelike curve in $X$. To show that $\gamma$ intersects $S''$, by Lemma~\ref{lem:prop_cauchy_set} we can assume that it intersects the Cauchy set $S_s$ (see Lemma \ref{lem:parallel_cauchy}) in a point $p \in S_s \cap I^+(S_t')$. Then $\gamma$ has to intersect $S_t'$ in a point $q \ll p$, which lies in $I^-(S_s) \cap S_t' \subset S''$. 

Now assume by contradiction that an inextendible  timelike curve $\gamma$ intersects $S''$ in two distinct points $p,q$. Without loss of generality we can assume that $p\ll q$. Since both $S_s$ and $S'_t$ are Cauchy sets by Lemma~\ref{lem:parallel_cauchy}, we can further assume that $p \in S_s \backslash S_t'$ and $q \in S_t' \backslash S_s$. The case $p \in S_t' \backslash S_s$ and $q \in S_s \backslash S_t'$ is analogous. Then $q \in J^-(S_s)$ as $p,q \in S''$, so there is an $r\in S_s$ such that $q\leq r$. This implies $$p\ll q\leq r\,,$$ so $p\ll r$ by the push-up principle. By our assumptions and Corollary~\ref{cor:inex_timelike} we can thus construct an inextendible timelike curve through $p,r \in S_s$ in contradiction to the fact that $S_s$ is a Cauchy set.

Similarly, if $S$, $S'$ are acausal and $\gamma$ is an inextendible causal curve in $M$, then in the previous contradiction argument, which now uses Lemma~\ref{lem:acausal_parallel_Sets}, with two distinct intersection points $p,q$, we find $$p\leq q\leq r\,,$$ so $p\leq r$ for $p,r\in S_s$. By our assumptions we can construct an inextendible causal curve through $p$ and $r$. As $S_s$ is acausal, this implies $p=r$ and hence $p=q$ by causality of $X$, a contradiction.

It remains to show that $d_J(S,S'')=s$ and $d_J(S'',S')=t$ for Cauchy sets $S,S'$. After recalling the definition of $N(A.\varepsilon)$ in \eqref{eq:nghb}, we claim that this is equivalent to 
\begin{equation}\label{eq:contained}
 S'' \subset N(S,s) \cap N(S',t)\,. 
\end{equation} 
Indeed, by definition \eqref{eq:contained} is equivalent to $d_J(S,S'')\leq s$ and $d_J(S'',S')\leq t$. In this case, since $S''$ is a Cauchy set, the triangle inequality for $(\mathcal C_X,d_J)$ gives 
$$d_J(S,S')\leq d_J(S,S'')+d_J(S'',S')\leq  s+t = d_J(S,S')$$ and so  $d_J(S,S'')=s$ and $d_J(S'',S')=t$ as claimed.

Finally, we establish \eqref{eq:contained}. By definition, we have $S_s \cap J^-(S'_t) \subset N(S,s)$. Suppose there is some $p \in ( S_s \cap J^-(S_t')) \backslash N(S',t)$. Then for every $\varepsilon>0$ there is an $x \in S\cap J^-(p)$ with $d(x,p) \geq s- \varepsilon$ as $p\in S_s$, and there is a $y \in S'$ with $d_J(p,y) >t \geq 0$ as $p\notin N(S',t)$. There is an inextendible timelike curve through $p$ and $y$, and it intersects $S_t'$ in some point $y'$ by Lemma~\ref{lem:parallel_cauchy} with $d(y,y') \leq t$. Since $p \in J^-(S_t')$, we see that $p\leq y'$. By the reverse triangle inequality and $d_J(p,y)>0$, we infer that $p\leq y$. Now we choose $\varepsilon$, and thus $x$, such that $d(x,p)+d(p,y)>s+t$. The reverse triangle inequality then implies
\[
 d_J(S,S') \geq d_J(x,y) \geq d_J(x,p) + d_J(p,y) > s + t \, ,
\]
a contradiction. Thus, we proved $S_s \cap J^-(S'_t) \subset N(S,s)\cap N(S',t)$. Likewise, we obtain $S_t' \cap J^-(S_s)  \subset N(S,s)\cap N(S',t)$ and hence \eqref{eq:contained}, as claimed.
\end{proof}

As a consequence of Proposition~\ref{prp:conditions_compl_metric}, Lemma~\ref{lem:mid}, and \cite[Theorem 2.4.16]{BBI01}, we see that if a Lorentzian pre-length space $X$ satisfies the standing assumptions and the Lorentzian distance to each Cauchy set is continuous, then the space of Cauchy sets $(\mathcal{C}_X,d_J)$ is a geodesic complete extended metric space. In particular, this applies in the setting of a spatially compact, timelike Cauchy complete, globally hyperbolic KS-Lorentzian length space. Indeed, in this case we already know from Lemma~\ref{lem:continuity_distance_fct} that the distance function $d_J(S,\cdot)$ is continuous.

\begin{cor}\label{cor:geodesic} For a spatially compact, timelike Cauchy complete, globally hyperbolic KS-Lorentzian length space $X$, the space $(\mathcal{C}_X,d_J)$ is a complete geodesic metric space. 
\end{cor}

Next, we observe that under suitable assumptions parallel Cauchy sets can be realized as level sets of a Cauchy time function.

\begin{prp}\label{prp:geodesic_time_function} Assume the standing assumptions on a Lorentzian pre-length space $X$ as given at the beginning of this subsection with empty chronological boundary. Let $S$ be a Cauchy set in $X$. If $d_J(S,\cdot)$ is continuous, then the function $\tau_S: X \rightarrow \R$ defined by
\[
        \tau_S(x) = \begin{cases} d_J(S,x) & \text{ if } x\in J^+(S)\,, \\
        -d_J(S,x) & \text{ if } x\in J^-(S) \end{cases}
\]
is continuous, has level sets $(\tau_S)^{-1}(t)=S_t$, and satisfies
\[
    \lim_{t\searrow \inf I} (\tau_S\circ \gamma)(t) =-\infty \quad \text{and}\quad\lim_{t\nearrow \sup I} (\tau_S\circ \gamma)(t) =\infty \, 
\]
for every inextendible timelike curve $\gamma:I \rightarrow X$.
\end{prp}
\begin{proof} Continuity of $\tau_S$ follows from continuity of $d_J(S,\cdot)$. Moreover, by construction the level sets of $\tau_S$ are the parallel sets of $S$, and these are Cauchy sets by Lemma~\ref{lem:parallel_cauchy}. Let $\gamma: I \rightarrow X$ be an inextendible timelike curve in $X$. By Lemma~\ref{lem:no_bubbling_(t)-inext} and emptiness of the chronological boundary, it does not have a future endpoint. Since $X$ satisfies Condition B, for every $t_0 \in I$ the function $d_J(\gamma(t_0),\gamma(t))$ is unbounded on both $I\cap (-\infty,t_0]$ and $I\cap [t_0,\infty)$.  We can assume that $\gamma(0) \in S$. The claim now follows from $d_J(S,\gamma(t)) \geq  d_J(\gamma(0),\gamma(t))$ and an analogous past argument.
\end{proof}

The conclusion of the proposition should be compared with Definition~\ref{dfn:strong_cauchy_time_function} of a strong Cauchy time function. In fact, a similar argument as in the proof of Proposition~\ref{prp:geodesic_time_function} shows that if $S$ is an acausal Cauchy hypersurface in a smooth, globally hyperbolic Lorentzian manifold satisfying Condition A$^*$, then $\tau_S$ is a strong Cauchy time function with the parallel sets of $S$ as level sets.
\vspace{0.2cm}

Finally, we show that the Cauchy time function from Proposition~\ref{prp:geodesic_time_function} resp.~parallel Cauchy sets give rise to geodesic rays in the space of Cauchy sets. We set $t_S\coloneqq-\sup\{ d_J(x,S) \mid x \in J^-(S) \}$ and $t^S\coloneqq \sup\{ d_J(S,x) 
\mid x \in J^+(S) \}$. 

\begin{lem} Assume the standing assumptions on a Lorentzian pre-length space $X$ as given at the beginning of this subsection.  Let $S$ be a Cauchy set in $X$. If $d_J(S,\cdot)$ is continuous, then $[0,t^S) \ni t \mapsto S_t$ is a geodesic in $(\Cauchy_X,d_J)$, 
i.e.~it holds that $$d_J(S_t,S_{s})=|t-s|$$ for all  $t,s \in [0,t^S)$. An analogous statement holds for the past parallel sets. 
\end{lem}
\begin{proof}
Let $t,s \in [0,t^S)$. The case $t,s\in (t_S,0]$ follows analogously. Assume without loss of generality $0<t<s$. 

For any $\varepsilon>0$ there are $x\in S$ and $z \in S_s$ with $d(x,z)+ \varepsilon \geq d_J(S,S_s) = s $. Note that the last inequality holds as $s<t^S$. We can assume that $d(x,z)>t$. By continuity a maximal causal curve from $x$ to $z$ intersects $S_t$ in some point $y$ with $d(x,y)+d(y,z) = d(x,z)$. Hence,
\[
 d_J(S_t,S_s) \geq d(y,z) = d(x,z)-d(x,y) \geq s-  
 \varepsilon -d_J(S,S_t)  =s-t -\varepsilon\,.
\]
Sending $\varepsilon$ to $0$ gives $d_J(S_t,S_s)\geq s-t$. 

Now assume that $d_J(S_t,S_s)>s-t$. Then there are $v \in S_t$ and $w \in S_s$ with $d(v,w)>s-t$. In particular, we have $v \notin \partial^+X$, and thus $d(S,v)=t$. We can choose some $\lambda\in (0,1)$ and some $u \in S$ with $d(u,v) > t - \lambda (d(v,w)-s+t)>0$. The reverse triangle inequality now implies that
\[
   d(u,w) \geq d(u,v) + d(v,w) > t - \lambda (d(v,w)-s+t) +d(v,w)>s \, ,
\]
a contradiction to $d_J(S,S_s)=s$. 
\end{proof}

\subsection{Local compactness}\label{sub:local_cpct} 
Recall that a topological space is called locally compact if every point has a compact neighborhood. For instance, the space of real-valued $1$-Lipschitz functions on an interval in $\R$ is locally compact with respect to the topology of uniform convergence if and only if the interval is bounded. This can be seen by the Arzel\`a--Ascoli theorem in the bounded case and by a sequence of bump functions with the bump moving to infinity in the unbounded case.
Similarly, we can see that the space of Cauchy hypersurfaces in the conical Minkowski spacetime over a compact domain is locally compact in view of Lemma~\ref{lem:cauchy_graph} and that it is not locally compact in Minkowski space. In particular, this shows that the spatial compactness assumption in the following result, which can be regarded as an analog of the \emph{Blaschke selection theorem}, is necessary. Recall that a Lorentzian manifold is spatially compact if it has a compact Cauchy hypersurface. 

\begin{prp} \label{prp:complete_metric_space} 
Let $M$ be a smooth, spatially compact, globally hyperbolic Lorentzian manifold. Then the metric space $(\Cauchy_M,d_J)$ is locally compact.
\end{prp}
\begin{proof} Let $S$ be a smooth, spacelike, compact Cauchy hypersurface and $X$ a smooth unit-norm timelike vector field on $M$; see \cite{BS03}. Since the flow lines of $X$ are neither future nor past imprisoned, they are inextendible by Lemma~\ref{lem:inextendible_charac} and Lemma~\ref{lem:no_bubbling_(t)-inext}.
Hence, each flow line intersects $S$ exactly once by definition. If $U\subset \R\times M$ denotes the open set where the maximal flow associated to $X$ is defined, the restricted flow $\Phi: U \cap (\R \times S) \rightarrow M$ is therefore a smooth bijection. As the spacelike Cauchy hypersurface $S$ is transverse to $X$, the map $\Phi$ is further a local diffeomorphism and hence a diffeomorphism.
  Moreover, we see that any other Cauchy hypersurface $S'$ can be represented as a graph $S'=S_f$ over $S$ via $\Phi$, i.e.~$S_f= \{ \Phi_{f(x)}(x) \mid x \in S\}$ for some continuous function $f:S\rightarrow \R$.

By the normalization assumption on $X$, we have that $L(\Phi_{\cdot}(x)_{|[t_1,t_2]})=t_2-t_1$ for all admissible $t_1<t_2$. This implies 
\begin{equation}
\label{eq:sup_norm_estimate}
     \|f-g \|_{\infty} \leq d_J(S_f,S_g) .
\end{equation}
For sufficiently small $r>0$ the set
\[
    \{ x \in M \mid d_J(x,S) \leq r \}\] 
can be expressed as an $r$-tubular neighborhood, that is, the image of the vectors in the normal bundle of $S$ with norm $\leq r$ under the normal exponential map, and hence it is compact. It contains all Cauchy hypersurfaces in the closed ball
\[
    \bar B_r(S)=\{S' \in \Cauchy_M \mid d_J(S,S') \leq r \}
\]
in $(\mathcal{C}_M,d_J)$. By compactness and smoothness of the Lorentzian metric, the functions $f$ representing the Cauchy hypersurfaces contained in $\bar B_r(S)$ are uniformly Lipschitz continuous. 
Let $S_{f_k}$ be a sequence of such Cauchy hypersurfaces. Since the functions $f_k$ are uniformly bounded by \eqref{eq:sup_norm_estimate} and uniformly Lipschitz continuous, we can assume that the sequence $f_k$ converges uniformly to a Lipschitz continuous function $f$ by the Arzel\`a--Ascoli theorem. Using the uniform Lipschitz continuity, we see that $S_f$ coincides with the set of all points $x\in M$ such that for any neighborhood $U$ of $x$ in $M$ it holds that $U\cap S_{f_k} \neq \emptyset$ for infinitely many $k$. 
We claim that $d_J(S_{f_k},S_f) \rightarrow 0$. Otherwise, there are some $\varepsilon>0$ and sequences $(x_j)_{j\in\N}$ and $(y_j)_{j\in\N}$ in $S$ with
\[
    d_J(\Phi_{f(x_j)}(x_j),\Phi_{f_j(y_j)}(y_j)) \geq \varepsilon \,.
\]
By compactness of $S$ we can assume that $x_j \rightarrow x \in S$ and $y_j \rightarrow y \in S$. By continuity this implies 
\[
    d_J(\Phi_{f(x)}(x),\Phi_{f(y)}(y)) \geq \varepsilon \,.
\]
Again by continuity, this implies
\[
    d_J(S_{f_k},S_{f_k}) \geq d_J(\Phi_{f_k(x)}(x),\Phi_{f_k(y)}(y)) \geq \varepsilon/2
\]
for sufficiently large $k$, a contradiction, as the 
hypersurfaces $S_{f_k}$ satisfy $d_J(S_{f_k},S_{f_k})=0$ by Lemma~\ref{lem:separation}, (ii).

It remains to show that $S_f$ is a Cauchy hypersurface. For an inextendible timelike curve $\gamma:I \rightarrow M$ in $M$, parametrized by arclength with respect to a complete auxiliary Riemannian metric, let $x_k=\gamma(t_k)$ be the intersection of $\gamma$ with $S_{f_k}$. Since the $x_k$ are contained in a compact set, they converge along a subsequence 
to some $x\in S_f$.  Since the inextendible curve $\gamma$ is not imprisoned, it leaves every compact set. Thus, also the $t_k$ stay in a compact set and hence converge along a subsequence to some $t\in I$ with $\gamma(t)= x \in S_f$ by continuity.

A globally hyperbolic Lorentzian manifold satisfies all assumptions of Proposition~\ref{prp:conditions_compl_metric} except that the inextendible timelike curves cannot be chosen Cauchy complete in general. The argument in the last paragraph of Proposition~\ref{prp:conditions_compl_metric} shows that every inextendible timelike curve intersects the $d_J$-limit $S_f$ of the Cauchy hypersurfaces $S_{f_k}$ only once.
\end{proof}

\subsection{Proof of Theorem~\ref{thm:metric_properties}} 
By Corollary~\ref{cor:inex_timelike}, we can apply Proposition~\ref{prp:conditions_ext_metric}, (ii), to deduce that $(\mathcal C_M, d_J)$ is an extended metric space, which is non-empty by Geroch's existence result \cite[Theorem 11]{Ge70}. If $M$ is timelike geodesically complete, then every inextendible timelike geodesic, which has constant speed by the geodesic equation, has infinite length in the future and past directions and is thus Cauchy complete, and hence by Proposition~\ref{prp:conditions_ext_metric}, (ii), we know that $(\mathcal C_M^w,d_J)$ is an extended metric space as well. 

Claim (i) is a consequence of Proposition~\ref{prp:conditions_compl_metric}. 
 Claim (ii) follows by the same proposition using Theorem~\ref{thm:complete}, (ii). 
Claim (iii) follows from Corollary~\ref{cor:compact_finite} and Proposition~\ref{prp:complete_metric_space}. Finally, Claim (iv) is a consequence of the previous claims and the (proof of the) Hopf--Rinow theorem for locally compact length spaces; see \cite[Theorem 2.5.28]{BBI01}.
\appendix

\section{Cauchy time functions for Lorentzian metric spaces}\label{app:time_functions}

Here we explain that the proof of \cite[Theorem~21]{Mi26}, in particular, generalizes to provide a strong Cauchy set in every separable, first-countable Lorentzian metric space $X$ with compact causal diamonds, empty bubbling boundary, and empty spacelike boundary. To this end, we first generalize the notion of a Cauchy time function as given in \cite[Definition~16]{Mi26} to the setting with boundary as follows. In order to distinguish this notion from the Cauchy time function constructed in Proposition~\ref{prp:geodesic_time_function}, we refer to it as a \textit{strong} Cauchy time function.

\begin{dfn} \label{dfn:strong_cauchy_time_function}
A continuous function $\tau : X \rightarrow [-\infty,\infty]$ is called \emph{time function} if $\tau(x) < \tau(y)$ whenever $x\leq y$ and $x\neq y$. It is called \emph{strong Cauchy time function} if in addition for every inextendible causal curve $\gamma: I \rightarrow X$ defined on an interval $I\subset \R$ it holds that
\[
    \lim_{t\searrow \inf I} (\tau\circ \gamma)(t) =-\infty \quad \text{and}\quad\lim_{t\nearrow \sup I} (\tau\circ \gamma)(t) =\infty \, .
\]
\end{dfn}

Note that strong Cauchy time functions can only exist if no constant curve on a degenerate interval is inextendible. This is for instance the case if $X$ is causally path-connected and has empty spacelike boundary. Note further that if there is an inextendible causal curve passing through a point $x \in X$, then a strong Cauchy time function $\tau$ on $X$ satisfies  $\tau(x)= \pm \infty$ if and only if $x \in \partial_{ca}^{\pm} X$. It is clear from the definition and the intermediate value theorem that every level set of a finite value of a strong Cauchy time function on $X$ is a strong Cauchy set in $X$.

In a causally path-connected, 
countably generated Lorentzian metric space, every causal curve is, 
after potential reparametrization, the restriction of an inextendible causal curve; see \cite[Proposition~20]{Mi26}. Going through the proof of this statement shows that it also holds in the setting of the following theorem 
(with ``inextendible'' in the sense of Definition~\ref{dfn:inext_def}). However, as we only apply the theorem in the setting of Corollary~\ref{cor:metric_corollary}, where we have timelike extensions and where  inextendible causal extensions are thus more easily guaranteed by Corollary~\ref{cor:inex_timelike}, we omit the details here.  

\begin{thm}\label{thm:cauchy_time_functions_exist}

Every separable, first-countable, causally path-connected Lorentzian metric space $X$ with compact causal diamonds, empty bubbling boundary, and empty spacelike boundary has a strong Cauchy time function.
\end{thm}

Examples that satisfy the conditions of this theorem and are not covered by \cite{Mi26} (due to their non-empty chronological boundary) are conical Minkowski spacetimes over a compact domain and the closed upper halfspace in $\Lo^{n+1}$.

Under the assumptions of the theorem,
the complement of the chronological boundary of $X$ is a separable, first-countable Lorentzian metric space with compact causal diamonds and empty boundary. A possible strategy would be to show that \cite[Theorem~21]{Mi26} holds under these assumptions and to extend the resulting Cauchy time function by $\pm \infty$ to $\partial^\pm X = \partial^\pm_{ca} X = \partial^\pm_{ch} X$. Since we have to go through the arguments of \cite[Theorem~21]{Mi26} anyway, we instead directly work in the setting of Theorem \ref{thm:cauchy_time_functions_exist} without removing the boundary before.

In the following, a map $f:X\rightarrow \R$ is called \emph{(anti-)isotone}, if $f(x) \leq f(y)$ ($f(x) \geq f(y)$) for all $x,y \in X$ with $x\leq y$. 

\begin{proof}[Proof of Theorem~\ref{thm:cauchy_time_functions_exist}]

The proof follows the same lines as the proof of \cite[Theorem~21]{Mi26} by Minguzzi. We recall the main arguments and point out the necessary modifications for the convenience of the reader.

The first step is to construct a time function. By the same proof as in \cite[Proposition~1.11]{MS24} and separability, there is a countable dense subset $\mathscr{S}\subset X$ that distinguishes points, that is, for all $x,y \in X$ with $x\neq y$ there exists some $z\in \mathscr{S}$ such that either $d(z,x)\neq d(z,y)$ or $d(x,z) \neq d(y,z)$. We write $\mathscr{S}=(z_k)_{k\in \N}$. The functions
\[
    f(x) \coloneqq \sum_{k=1}^\infty \frac{1}{2^k} \frac{d(z_k,x)}{1+d(z_k,x)} \quad \text{and} \quad g(x) \coloneqq \sum_{k=1}^\infty \frac{1}{2^k} \frac{d(x,z_k)}{1+d(x,z_k)}
\]
satisfy the same properties as in \cite{Mi26} except that they take the value $0$ at the chronological boundary, more precisely, $f: X \rightarrow [0,1]$ is continuous and isotone with $f(x)=0 $ if and only if $ x \in \partial^-X$, and $g: X \rightarrow [0,1]$ is continuous and anti-isotone with $g(x)=0$ if and only if $x \in \partial^+X$. As $\partial^+X \cap \partial^-X= \emptyset$ by assumption, we obtain a well-defined isotone continuous map $\tau\coloneqq \ln(f/g): X \rightarrow [-\infty,\infty]$ with $\tau(x)=\pm \infty$ if and only if $x \in \partial^\pm X$. 

Since $\mathscr{S}$ distinguishes points in $X$, for any pair of points $x\leq y$ with $x\neq y$ in $X$, there is some $z_k \in \mathscr{S}$ that distinguishes them. We suppose that $d(z_k,x) \neq d(z_k,y)$; the other case is analogous. By the reverse triangle inequality, we necessarily have $d(z_k,x) < d(z_k,y)$, which implies that $f(x) < f(y)$ and $g(x) \geq g(y)$. As $x\leq y$ with $x\neq y$,
we have $x \notin \partial^+ X$ and thus $g(x)>0$. It follows that $\tau(x) < \tau (y)$, i.e.~$\tau$ is a time function.

To prove that $\tau$ is actually a strong Cauchy time function, we consider a future inextendible causal curve $\gamma: I \rightarrow X$ and show that $\lim_{t \nearrow\sup I} \tau(\gamma(t))= \infty$. The corresponding past statement follows 
analogously. If $\gamma$ has a future endpoint, then, by inextendibility of $\gamma$ and causal path-connectedness of $X$, it must lie in the future causal boundary $\partial^+ X$, where $\tau$ attains the value $\infty$. Hence, we can assume that $\gamma$ has no future endpoint. We can moreover assume that $I=[0,\infty)$ by reparametrization. In this case, the argument in the proof of \cite[Theorem~21]{Mi26} shows that $\lim_{t\rightarrow \infty} g(\gamma(t))= 0$. Indeed, by Lemma~\ref{lem:inextendible_charac} 
the curve $\gamma$ is not future imprisoned. Therefore, $\gamma$ eventually leaves every (compact) causal diamond $J^+(\gamma(0)) \cap J^-(z_k)$  
and hence also $J^-(z_k)$ forever, so that each summand of $g\circ \gamma$ 
becomes eventually zero forever, which due to the geometric decay given by $2^{-k}$ implies $g(\gamma(t))\to 0$ as $t\to\infty$. Since $f(x)$ is bounded from below for sufficiently large $t$, $\lim_{t\nearrow \sup I} \tau(\gamma(t))= \infty$, and thus $\tau$ is a strong Cauchy time function as claimed.
\end{proof}

Note that we used first-countability and causal simplicity of $X$ when applying Lemma~\ref{lem:inextendible_charac}. In contrast, in the proof of \cite[Theorem~21]{Mi26}, local compactness of a countably generated Lorentzian metric space (see \cite[Proposition 3.3]{BMS25}) is used in this step; see \cite[Proposition 18, Proposition 19]{Mi26}.

\section{Properties of Cauchy sets}\label{app:continuity_Lorentzian_distance}

\subsection{Continuity of the Lorentzian distance to a Cauchy set}
Here we provide the details of the proof of Lemma~\ref{lem:continuity_distance_fct}.

\begin{proof}[Proof of Lemma~\ref{lem:continuity_distance_fct}]
By definition of $d_J(S,p)$, $p\in X$, for any $\varepsilon>0$ there is some $q \in S$ with $d_J(S,p) - \varepsilon < d_J(q,p)$. By lower semicontinuity of the Lorentzian distance, this implies 
\[
d_J(S,p) - \varepsilon \leq  \liminf_{x\rightarrow p} d_J(q,x) \leq \liminf_{x\rightarrow p} d_J(S,x)
\]
and so $d_J(S,p) \leq \liminf_{x\rightarrow p} d_J(S,x)$ by sending $\varepsilon$ to $0$. 

(i) Suppose that $S$ is in addition compact (and hence sequentially compact). By contradiction assume that we have $\limsup_{x \rightarrow p} d_J(S,x) >d_J(S,p)$. Then there is some $\varepsilon>0$ and a sequence $p_n \rightarrow p$ as $n\to \infty$ with \[\limsup_{n\rightarrow \infty}  d_J(S,p_n) > d_J(S,p)+ \varepsilon\,.\] After passing to a subsequence, let $q_n\in S$ be a sequence with $d_J(q_n,p_n) + 1/n \geq d_J(S,p_n)$. By sequential compactness of $S$ we can assume that $q_n$ converges to some $q \in S$. By continuity of the Lorentzian distance, we obtain $d_J(q,p) > d_J(S,p)$, a contradiction.

(ii) Suppose now that $X=M$ is a smooth Lorentzian manifold and that $S$ is an acausal Cauchy hypersurface. We fix an auxiliary complete Riemannian background metric $h$ on $M$; cf.~\cite{NO61}. By \cite[Corollary 3.47]{Mi19} and Lemma~\ref{lem:prop_cauchy_set} for every $p\in M$ there is a maximal (continuous) causal curve from $p$ to $S$ with Lorentzian length $d(p,S)$. We replace it by a smooth maximal causal curve with the same endpoints, and subsequently extend it to an inextendible causal curve $\gamma$; see Remarks on Lemma~\ref{rem:rem_on_inext}, (a). In fact, by Remarks on Lemma~\ref{rem:rem_on_inext}, (b), we can assume that $\gamma$, when parametrized by $h$-arclength, is defined on $(-\infty,\infty)$. Let $p\in M$ and let $p_n \in M$ be a sequence converging to $p$. By Lemma~\ref{lem:prop_cauchy_set}, we can assume that $p_n \in I^-(S)$ as $d_J(S,S)=0$. By the previous argument, there is a sequence of inextendible causal curves $\gamma_n:(-\infty,\infty) \rightarrow M$  parametrized by $h$-arclength with $\gamma_n(0)=p_n$, $\gamma_n(a_n)\in S$ for some $a_n > 0$, $d(p_n,\gamma_n(a_n))=d(p_n,S)$, and such that $\gamma_n$ is maximal on $[0,a_n]$. By the limit curve theorem in the formulation of \cite[Theorem 2.51]{Mi19}, the curves $\gamma_n$ converge up to a subsequence uniformly on compact sets to an inextendible causal curve $\gamma: (-\infty,\infty) \rightarrow M$ with $\gamma(0)=p$. 
This curve intersects $S$ by \cite[Lemma 14.29]{On83}, i.e.~there is some $a\in \R$ with $\gamma(a) \in S$. Since $S$ is acausal, there is a unique such $a$ and $\gamma(t) \in I^+(S)$ for all $t>a$. By passing to another subsequence,  we can assume that $a_n \rightarrow b \in [0,\infty]$. If $b= \infty$, choose a $t>a$ such that $\gamma(t)\in I^+(S)$; since $I^+(S)$ is open, local uniform convergence gives $\gamma_n(t)\in I^+(S)$ for all sufficiently large $n$. Taking $n$ larger if necessary, we have $a_n>t$ with $\gamma_n(a_n)\in S$, a contradiction. Hence, $b< \infty$. Uniqueness of $a$ and closedness of $S$ (see e.g. \cite[Proposition 5.9]{BG25}) imply that $a=b$. Hence, $\gamma_n(0)\rightarrow \gamma(0)=p$ and $\gamma_n(a_n)\rightarrow \gamma(a)$ as $n\to\infty$ by uniform convergence and  continuity of $\gamma$. Therefore,
\[
    d_J(p_n,S)=d_J(\gamma_n(0),\gamma_n(a_n)) \rightarrow d_J(p,\gamma(a)) \leq d_J(p,S) \, 
\]
by continuity of $d_J$, which shows upper semicontinuity of $d_J(S,\cdot)$.
\end{proof}

\subsection{Acausality of parallel Cauchy sets}
\label{sub:acausality_of_Cauchy_sets}

Here we provide a proof for Lemma~\ref{lem:acausal_parallel_Sets}.

\begin{proof}[Proof of Lemma~\ref{lem:acausal_parallel_Sets}] By Lemma~\ref{lem:continuity_distance_fct} the distance function $d_J(S,\cdot)$ is continuous, and thus all the parallel sets $S_t$ are Cauchy hypersurfaces by Lemma~\ref{lem:parallel_cauchy}. Since every Cauchy hypersurface in a smooth Lorentzian manifold is intersected by any inextendible causal curve (see \cite[Lemma 14.29]{On83}, for instance), it remains to show that no parallel set $S_t$ is intersected twice by a causal curve. 

Suppose a causal curve $\gamma$ connects two points $y,z \in S_t$ with $y \leq z$. We can assume without loss of generality that $t>0$. Then $y \notin \partial^+M$, and the reverse triangle inequality implies that
\[
    t=d_J(S,z) \geq  d_J(S,y) + d(y,z) \geq d_J(S,y)=t \, .
\]
Hence, $d(y,z)=0$ and $d_J(S,z)=d_J(S,y)$. Recall that $S$ is closed by Lemma~\ref{lem:prop_cauchy_set}. Therefore, by \cite[Corollary 3.47]{Mi19} there is a maximal causal curve from a point $x\in S$ to $y$ of Lorentzian length $d_J(S,y)$. The concatenation of this curve with a null geodesic from $y$ to $z$ is a causal curve of Lorentzian length \[d_J(x,y)=d_J(S,y)=d_J(S,z) \geq d_J(x,z) \geq d_J(x,y)+d_J(y,z)\geq d_J(x,y)\]from $x$ to $z$, i.e.~a maximal causal curve from $x$ to $z$ of length $d_J(x,z)=d_J(S,y)=t>0$. Since this curve is timelike and nonconstant between $x$ and $y$, it is thus timelike everywhere by the push-up principle in the form of \cite[Theorem 3.20]{KS18}. Consequently, $y=z$, and the claim follows.
\end{proof}


\end{document}